\documentclass[preprint,12pt]{elsarticle}
\usepackage{amssymb}
\usepackage{amsmath}
\usepackage{amsthm}
\usepackage{amsfonts}
\usepackage[all]{xy}
\usepackage{color}
\usepackage{filecontents}
\usepackage{natbib}
\usepackage{tikz}
\usetikzlibrary{positioning,automata}
\usepackage{lineno}
\usepackage{ mathrsfs }
\usepackage{cases}
\usepackage{pgfplots}

\usetikzlibrary{shapes.geometric}

\newtheorem{theorem}{Theorem}[section]
\newtheorem{proposition}[theorem]{Proposition}
\newtheorem{lemma}[theorem]{Lemma}

\theoremstyle{definition}
\newtheorem{definition}[theorem]{Definition}
\newtheorem{example}[theorem]{Example}

\newtheorem{remark}[theorem]{Remark}

\newcommand{\R}{\mathbb{R}}  
\newcommand{\N}{\mathbb{N}}  

\numberwithin{equation}{section}

\journal{~}

\begin{document}

\begin{frontmatter}

\title{Structure for $g$-Metric Spaces and Related Fixed Point Theorems }

\author[label1]{Hayoung Choi\corref{corr1}}
\cortext[corr1]{Corresponding author}
\ead{hayoung79choi@gmail.com}
\address[label1]{School of Information Science and Technology, ShanghaiTech University, Pudong district, Shanghai 200031, China}

\author{Sejong Kim \fnref{label2}}
\ead{skim@chungbuk.ac.kr}
\address[label2]{Department of Mathematics, Chungbuk National University, Cheongju 361-763, Republic of Korea}

\author{Seung Yeop Yang \fnref{label3}}
\ead{seungyeop.yang@du.edu}
\address[label3]{Department of Mathematics, University of Denver, Denver, CO 80208, USA}

\begin{abstract}
In this paper, we propose a generalized notion of a distance function, called a $g$-metric.
The $g$-metric with degree $n$ is a distance of $n+1$ points, generalizing the ordinary distance between two points and $G$-metric between three points. Indeed, it is shown that the $g$-metric with degree 1 (resp. degree 2) is equivalent to the ordinary metric (resp. the $G$-metric).
Fundamental properties and several examples for the $g$-metric are also given.
Moreover, topological properties on the $g$-metric space including the convergence of sequences and the continuity of mappings on the $g$-metric space are studied. Finally, we generalize some well-known fixed point theorems including Banach contraction mapping principle and \'Ciri\'c fixed point theorem in the $g$-metric space.
\end{abstract}


\end{frontmatter}


\section{Introduction and preliminaries}
A distance is a measurement how far apart each pair elements of a given set are.
The distance function in mathematics and many other scientific fields is a crucial concept. For instance,
the distance function is used to quantify a dissimilarity (or equivalently similarity) between two objects in some sense.
However, due to massive and complicated data sets today, the definition of a distance function is required to be generalized.

Numerous ways to generalize the notion of a distance function have been studied \cite{AKOH15,Khamsi2015}.
Among them, we consider the $G$-metric space, which allows us to establish many topological properties. Moreover, a variety of the fixed point theorems on the $G$-metric space have extensively been studied.
In this paper we give a generalized notion of a distance function between $n+1$ points, called a $g$-metric. It coincides with the ordinary distance between two points and with the $G$-metric between three points. Furthermore, we establish topological notions and properties on the $g$-metric space including the convergence of sequences and continuity of mappings. From these topology on the $g$-metric space we generalize some well-known fixed point theorems such as the Banach contraction mapping principle, weak contraction mapping principle, and \'Ciri\'c fixed point theorem.

Let $\N$ (resp. $\R$) be the set of all nonnegative integers (resp. all real numbers), and let $\R$ be the set of all real numbers. We denote as $\R_{+}$ the set of all nonnegative real numbers.
For a finite set $A$, we denote the number of distinct elements of $A$ by $n(A)$.

The definition of a distance function was proposed by M. Fr\'{e}chet \cite{Frechet06} in 1906.
\begin{definition} \label{D:distance}
Let $\Omega$ be a nonempty set. A function $d: \Omega \times \Omega \longrightarrow \R_{+}$ is
called a \emph{distance function} or \emph{metric} on $\Omega$ if it satisfies the following conditions:
\begin{itemize}
    \item[(1)] (identity) $d(x,y)=0$ if $x=y$,
    \item[(2)] (non-negativity) $d(x,y)>0$ if $x \neq y$,
    \item[(3)] (symmetry) $d(x,y) = d(y,x)$ for all $x, y \in \Omega$,
    \item[(4)] (triangle inequality) $d(x,y) \leq d(x,z) + d(z,y)$ for all $x, y, z \in \Omega$.
\end{itemize}
The pair $(\Omega, d)$ is called a \emph{metric space}.
\end{definition}

The first attempt to generalize the ordinary distance function to a distance of three points was introduced by Gahler \cite{Gahler63, Gahler66} in 1963.
\begin{definition} \label{D:Gahler}
Let $\Omega$ be a nonempty set. A function $d:\Omega\times \Omega \times \Omega \longrightarrow \R_+$ is
called a \emph{2-metric} on $\Omega$ if it satisfies the following properties:
\begin{itemize}
    \item[$(m1)$] For all $x,y \in \Omega$ with $x \neq y$, there is $z\in \Omega$ such that $d(x,y,z) \neq 0$,
    \item[$(m2)$] $d(x,y,z)=0$ if $x=y,$ $y=z,$ or $z=x,$
    \item[$(m3)$] $d(x,y,z) = d(x,z,y) = \cdots$ (symmetry in all three variables),
    \item[$(m4)$] $d(x,y,z) \leq d(x,y,w) + d(x,w,z) + d(w,y,z)$ for all $x, y, z, w \in \Omega$.
\end{itemize}
The pair $(\Omega, d)$ is called a \emph{2-metric space}.
\end{definition}

\begin{example}\label{ex:2-metric}
We give a few examples of a $2$-metric.
\begin{itemize}
\item[(1)] Let $d(x,y,z)$ be the area of the triangle with vertices at $x,y,z\in \R^2$. Then $d$ is a 2-metric on $\R^{2}$ \cite{MS06}.
\item[(2)] For a given metric $\delta$ on a set $\Omega$ with $n(\Omega) \geq 3$, define $d: \Omega \times \Omega \times \Omega \longrightarrow \R_+$ given by
$$ d(x,y,z) = \min\{ \delta(x,y), \delta(y,z), \delta(z,x) \}, $$
for all $x,y,z\in \Omega$.
Then $(\Omega, d)$ is a 2-metric space. The proof is left to the reader.

\item[(3)] Let $\Omega$ be a set with $n(\Omega) \geq 3$. Define $d: \Omega \times \Omega \times \Omega \longrightarrow \R_+$ given by
\begin{displaymath}
d(x,y,z) =
\begin{cases}
    0, & \hbox{if $x=y,$ $y=z,$ or $z=x$;}\\
    1, & \hbox{otherwise.}
\end{cases}
\end{displaymath}
Then, it is easy to check that $d$ is a 2-metric on $\Omega$.
\end{itemize}
\end{example}

It was shown that a 2-metric is not a generalization of the usual notion of a metric
\cite{HCW88}. Dhage in 1992 introduced a new class of generalized metrics called D-metrics  \cite{dhage92}.

\begin{definition}
Let $\Omega$ be a nonempty set. A function $D: \Omega \times \Omega \times \Omega \longrightarrow \R_+$ is
called a \emph{D-metric} on $\Omega$ if it satisfies the following conditions:
\begin{itemize}
    \item[(D1)] $D(x,y,z)=0$ if and only if $x=y=z$,
    \item[(D2)] $D(x,y,z) = D(x,z,y) =\cdots $ (symmetry in all three variables),
    \item[(D3)] $D(x,y,z) \leq D(x,y,w) + D(x,w,z) + D(w,y,z)$ for all $x,y,z,w\in \Omega,$
    \item[(D4)] $D(x, y, y) \leq D(x, z, z) + D(z, y, y)$ for all $x, y, z \in \Omega$.
\end{itemize}
The pair $(\Omega,D)$ is called a \emph{D-metric space}.
\end{definition}

\begin{example}\label{example:D-metric}
For any given metric space $(\Omega, \delta)$ the following are $D$-metrics \cite{dhage92}.
\begin{itemize}
    \item[(1)] $\displaystyle D(x, y, z) = \frac{1}{3} (\delta(x, y) + \delta(y, z) + \delta(x, z)).$
    \item[(2)] $D(x, y, z) = \max\{ \delta(x, y), \delta(y, z), \delta(x, z) \}.$
\end{itemize}
\end{example}

Topological structures and fixed points in a $D$-metric space have been studied.
However, several errors for fundamental topological properties in a $D$-metric space were found \cite{MS03,NRS05}.
Due to these considerations, Mustafa and Sims \cite{MS06} proposed a more appropriate notion of a generalized metric space.
For more information, see \cite{AKOH15} and references therein.

\begin{definition}
Let $\Omega$ be a nonempty set. A function $G:\Omega\times \Omega \times \Omega \longrightarrow \R_+$ is
called a \emph{G-metric} on $\Omega$ if it satisfies the following conditions:
\begin{itemize}
    \item[(G1)] $G(x,y,z)=0$ if $x=y=z$,
    \item[(G2)] $G(x,x,y) > 0$ for all $x,y \in \Omega$ with $x \neq y$,
    \item[(G3)] $G(x, x, y) \leq G(x, y, z)$ for all $x, y, z \in \Omega$ with $y \neq z$,
    \item[(G4)] $G(x,y,z) = G(x,z,y) = \cdots $ (symmetry in all three variables $x,y,z$),
    \item[(G5)] $G(x,y,z) \leq G(x,w,w) + G(w,y,z)$ for all $x,y,z,w \in \Omega$.
\end{itemize}
The pair $(\Omega,G)$ is called a \emph{G-metric space}.
A $G$-metric space $(\Omega,G)$ is said to be \emph{symmetric} if
\begin{itemize}
    \item[(G6)] $G(x,y,y) = G(x,x,y)$ for all $x,y\in \Omega$.
\end{itemize}
\end{definition}

\begin{example}
The following are G-metrics.
\begin{itemize}
\item[(1)] Let $d(x,y,z)$ be the perimeter of the triangle with vertices at $x,y,z$ on $\R^2$. Then $(\R^2,d)$ is a $G$-metric space. More generally, for a given normed space $(\Omega, \| \cdot \|)$, define
$$ d(x,y,z)= c(\|x-y \|+ \|y-z \| + \|z-x \|) $$
for $x,y,z \in \Omega$ with a fixed $c>0$. Then by Theorem \ref{thm:relation} (2) and Example \ref{Ex:average g-metric} it follows that $d$ is a $G$-metric.
\item[(2)] Let $\Omega = \{ x,y \}$ and let $G(x,x,x)=G(y,y,y)=0$, $G(x,x,y)=1,G(x,y,y)=2$ and extend $G$ to all of $\Omega\times \Omega \times \Omega$ by symmetry in the variables. Then $G$ is a $G$-metric which is not symmetric  \cite{MS06}.
\item[(3)] By Theorem \ref{thm:relation} (2) and Example \ref{Ex:max g-metric}, it is clear that Example \ref{example:D-metric} (2) is a $G$-metric.
\end{itemize}
\end{example}

\section{Theory of a $g$-metric}

Now we propose a new definition of a generalized metric for $n$ points instead of two or three points in  a given set.
For a set $\Omega$, we denote $\displaystyle \Omega^n := \prod_{i=1}^n \Omega.$
\begin{definition}\label{D:g-metric space}
Let $\Omega$ be a nonempty set. A function $g:\Omega^{n+1} \longrightarrow \R_+$ is
called a \emph{$g$-metric with order $n$}
on $\Omega$ if it satisfies the following conditions:
\begin{itemize}
    \item[$(g1)$](positive definiteness) $g(x_0, \ldots, x_{n})=0$ if and only if $x_0 = \cdots = x_{n}$,
    \item[$(g2)$](permutation invariancy) $g(x_0, \ldots, x_{n}) = g(x_{\sigma(0)}, \ldots, x_{\sigma(n)})$ for any permutation $\sigma$ on $\{ 0, 1, \ldots, n \}$,
    \item[$(g3)$](monotonicity) $g(x_0, \ldots, x_{n}) \leq g(y_0, \ldots, y_{n})$ for all $(x_0, \ldots, x_{n}), (y_0, \ldots, y_{n}) \in \Omega^{n+1}$ with $\{ x_i : i = 0, \ldots, n \} \subsetneq \{ y_i : i = 0, \ldots, n \}$,
    \item[$(g4)$](triangle inequality) for all $x_0, \ldots, x_s, y_0, \ldots, y_t, w \in \Omega$ with $s + t + 1 = n$
$$ g(x_0, \ldots, x_s, y_0, \ldots, y_t) \leq g(x_0, \ldots, x_s, w, \ldots, w) + g(y_0, \ldots, y_t, w, \ldots, w). $$
\end{itemize}
The pair $(\Omega, g)$ is called a \emph{$g$-metric space}.
\end{definition}

\begin{definition}
A $g$-metric on $\Omega$ is called \emph{multiplicity-independent} if the following holds
$$ g(x_0, \ldots, x_{n}) = g(y_0, \ldots, y_{n}) $$
for all $(x_0, \ldots, x_{n}), (y_0, \ldots, y_{n}) \in \Omega^{n+1}$
with $\{ x_i : i=0, \ldots, n \} = \{y_i : i=0, \ldots, n \}$.
\end{definition}

Note that for a given multiplicity-independent $g$-metric with order 2, it holds that $g(x,y,y)=g(x,x,y)$.
For a given multiplicity-independent $g$-metric with order 3, it holds that $g(x,y,y,y) = g(x,x,y,y) = g(x,x,x,y)$ and $g(x,x,y,z) = g(x,y,y,z) = g(x,y,z,z)$. This is why we call it multiplicity-independent rather than symmetric  which was proposed in the $G$-metric.

\begin{remark}
If we allow equality under the condition of monotonicity in Definition \ref{D:g-metric space}, i.e., ``$g(x_0, \ldots, x_{n}) \leq g(y_0, \ldots, y_{n})$ for all $(x_0, \ldots, x_{n}), (y_0, \ldots, y_{n}) \in \Omega^{n+1}$ with $\{ x_i : i = 0, \ldots, n \} \subseteq \{ y_i : i = 0, \ldots, n \}$", then every $g$-metric becomes multiplicity-independent.
\end{remark}

Let us explain why the condition $(g4)$ can be considered as a generalization of the triangle inequality.
Recall that the triangle inequality condition for a distance function $d$ is $d(x,y) \leq d(x,z) + d(z,y)$ for all $x,y,z$.
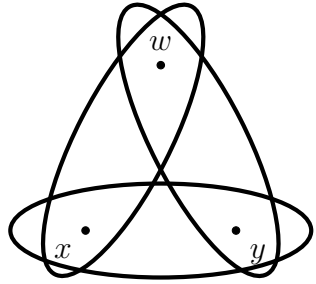
\begin{center}
\begin{tikzpicture}[
    dot/.style={draw, circle, inner sep=1pt, fill=black},
    group/.style={draw=#1, ellipse, ultra thick, minimum width=4cm, minimum height=1.25cm}
    ]
    \node[dot,label=below left:$x$] at (0,-0.2) (x3) {};
    \node[dot,white] at (1,-0.2) (x4) {};
    \node[dot,label=below right:$y$] at (2,-0.2) (x5) {};
    \node[dot,white] at (0.5,1) (x2) {};
    \node[dot,white] at (1.5,1) (x6) {};
    \node[dot,label=$w$] at (1,2) (x1) {};
    \node[group=black, label={[black]below:}] at (x4) (E2) {};
    \node[group=black, label={[black]above right:}, rotate=63] at (x2) (E3) {};
    \node[group=black, label={[black]above left:}, rotate=-63] at (x6) (E1) {};
\end{tikzpicture}
\end{center}
The point $w$ is required to measure approximately the distance between $x$ and $y$ with the distances between $x$ and $w$ and between $w$ and $y$. Note that one cannot measure the distance between $x$ and $y$ by the distances $d(x,w_1)$ and $d(y,w_2)$ with $w_1 \neq w_2$.
Consider $d(x,y)$ as a dissimilarity between $x$ and $y$. Clearly, if $x=y$, then the dissimilarity is 0, vice versa. Also, the dissimilarity between $x$ and $y$ is same as the dissimilarity between $y$ and $x$. If $x$ (resp. $y$) and $z$ (resp. $z$) are sufficiently similar, then by the triangle inequality $x$ and $y$ must be sufficiently similar.

In the similar way, one can generalize the definition of triangle inequality for the $g$-metric. Specifically,
one can see from the definition of triangle inequality for the $g$-metric that
if both $g(x_0, \ldots, x_s, w, \ldots, w)$ and $g(y_0, \ldots, y_t, w, \ldots, w)$ are sufficiently small, then
$g(x_0, \ldots, x_s, y_0, \ldots, y_t)$ must be sufficiently small.
That is, the higher similarities two data sets $\{ x_0, \ldots, x_s, w \}$ and $\{ y_0, \ldots, y_t, w \}$ have, the higher similarity data set
$ \{ x_0, \ldots, x_s, y_0, \ldots, y_t \} $ does. Note that $w$ is a necessary point to combine information about similarity for each data set.

The following theorem shows us that $g$-metrics generalize the notions of ordinary metric and $G$-metric.

\begin{theorem}\label{thm:relation}
Let $\Omega$ be a given nonempty set. The following are true.
\begin{itemize}
\item[(1)] $d$ is a $g$-metric with order 1 on $\Omega$ if and only if $d$ is a metric on $\Omega$.
\item[(2)] $d$ is a (resp. multiplicity-independent) $g$-metric with order 2 on $\Omega$ if and only if $d$ is a (resp. symmetric) $G$-metric on $\Omega$.
\end{itemize}
\end{theorem}

\begin{proof}
\begin{itemize}
\item[(1)] By the definition of $g$-metric, $d$ is a $g$-metric with order 1 if and only if it satisfies the following conditions.
\begin{itemize}
    \item[$(g1)$] $d(x,y)=0$ if and only if $x=y$,
    \item[$(g2)$] $d(x,y) = d(y,x)$,
    \item[$(g3)$] $0 = d(x,x) \leq d(x,y)$ if $x \neq y$,
    \item[$(g4)$] $d(x,z) \leq d(x,y) + d(y,z)$ for all $x,y,z \in \Omega$.
\end{itemize}
Clearly, these are equivalent to the axioms for a distance function.

\item[(2)] By the definition of $g$-metric, $d$ is a $g$-metric with order 2 if and only if it satisfies the following conditions.
\begin{itemize}
    \item[$(g1)$] $d(x,y,z)=0$ if and only if $x=y=z$,
    \item[$(g2)$] $d(x,y,z) = d(x,z,y) = \cdots $ (symmetry in all three variables),
    \item[$(g3)$] $d(x,x,y) \leq d(x,y,z)$ for all $x,y,z \in \Omega $ with $y \neq z$,
    \item[$(g4)$] $d(x,y,z) \leq d(x,w,w) + d(y,z,w) $ for all $x,y,z,w \in \Omega$.
 \end{itemize}
These are exactly same as the conditions of $G$-metric.
Thus, the $g$-metric with order 2 and the $G$-metric are identical. Moreover, one can see easily the equivalence between the multiplicity-independence of a $g$-metric and the symmetry of a $G$-metric.
\end{itemize}
\end{proof}

Remark that since a $g$-metric with order 2 on a nonempty set $\Omega$ is a $G$-metric, any $g$-metrics with order 2 satisfy all properties of the $G$-metric as shown in \cite{MS06}.
Moreover, if $d$ is a $g$-metric with order 2 on $\Omega$, then $d$ is also a $D$-metric.

Next, we show that an explicit form of conditions for a $g$-metric with order 3.
\begin{proposition}
$d$ is a $g$-metric with order 3 on $\Omega$ if and only if the following conditions hold:
\begin{itemize}
    \item[(1)] $g(x,y,p,q) = 0$ if and only if $x=y=p=q$.
    \item[(2)] $g(x,y,p,q) = g(y,x,p,q) = g(p,y,x,q) =g(q,y,p,x) = g(x,p,y,q) = g(x,y,q,p) $.
    \item[(3)]
$g(x,y,y,y) \leq g(x,x,y,p)$,\\
$g(x,y,y,y) \leq g(x,y,y,p)$,\\
$g(x,y,y,y) \leq g(x,y,p,p)$,\\
$g(x,x,y,y) \leq g(x,x,y,p)$,\\
$g(x,x,y,y) \leq g(x,y,y,p)$,\\
$g(x,x,y,y) \leq g(x,y,p,p)$,\\
$g(x,y,p,p) \leq g(x,y,p,q)$\\
    for all distinct $x,y,p,q\in \Omega$.
    \item[(4)] $g(x,y,p,q) \leq g(x,w,w,w) + g(y,p,q,w)$,\\
    $g(x,y,p,q) \leq g(x,y,w,w) + g(p,q,w,w)$
    for all $x,y,p,q,w\in \Omega$.
\end{itemize}
\end{proposition}
\begin{proof}
It is easy to show that the condition (1) (resp. (2)) is equivalent to $(g1)$ (resp. $(g2)$).
For $(g4)$ since $s+t=2$, there are three possibilities: (i) $s=0$ and $t=2$, (ii) $s=1$ and $t=1$, and (iii) $s=2$ and $t=0$. Since $x,y,p,q,w$ are arbitrary, (i) and (iii) are equivalent. Thus we have two different inequalities given in (4).
For $(g3)$ let $X=\{ x_0, x_1,x_2,x_3 \}$ and $Y=\{y_0, y_1,y_2,y_3  \}$ with $X \subsetneq Y$.
There are six possibilities: (i) $n(X)=1$ and  $n(Y)=2$; (ii) $n(X)=1$ and  $n(Y)=3$; (iii) $n(X)=1$ and  $n(Y)=4$; (iv) $n(X)=2$ and  $n(Y)=3$; (v) $n(X)=2$ and  $n(Y)=4$; (vi) $n(X)=3$ and  $n(Y)=4$.
Note that $(1)$ implies the condition $(g3)$ for the cases (i), (ii), and (iii). Also the condition $(g3)$ for (iv) and (vi) implies the condition $(g3)$ for (v).
Thus
the conditions $(g1),(g2),(g3)$, and $(g4)$ are equivalent the conditions  $(1),(2),(4)$, and (iv) and (vi).
For (iv), let $X=\{x,y\}$ and $Y=\{x,y,p\}$. Then, we have for all distinct $x,y,p\in \Omega$,
\begin{align*}
g(x,y,y,y) \leq g(x,x,y,p),\quad & g(x,y,y,y) \leq g(x,y,y,p),\\
g(x,y,y,y) \leq g(x,y,p,p),\quad & g(x,x,y,y) \leq g(x,x,y,p),\\
g(x,x,y,y) \leq g(x,y,y,p),\quad & g(x,x,y,y) \leq g(x,y,p,p).
\end{align*}

For (vi), let $X=\{x,y,p\}$ and $Y=\{x,y,p,q\}$. Then, we have
$g(x,y,p,p) \leq g(x,y,p,q)$ for all distinct $x,y,p,q\in \Omega$.
\end{proof}

A new $g$-metric can be constructed from given $g$-metrics.
\begin{lemma}\label{lemma:new_g-metric}
Let $(\Omega, g)$ and $(\Omega, \tilde{g})$ be $g$-metric spaces. Then the following functions, denoted by $d$, are $g$-metrics on $\Omega$.
\begin{itemize}
\item[(1)] $d(x_0, x_1, \ldots, x_{n}) = g(x_0, x_1, \ldots, x_{n}) + \tilde{g}(x_0, x_1, \ldots, x_{n})$.
\item[(2)] $d(x_0, x_1, \ldots, x_{n}) = \psi(g(x_0, x_1, \ldots, x_{n}))$ where $\psi $ is a function on $[0,\infty)$ satisfies
\begin{itemize}
\item[(i)] $\psi $ is increasing on $[0,\infty)$;
\item[(ii)] $\psi(0)=0 $;
\item[(iii)] $\psi (x+y)  \leq \psi (x) +  \psi (y)$ for all $x,y\in [0,\infty)$.
\end{itemize}
\end{itemize}
\end{lemma}

\begin{proof}
\begin{itemize}
\item[(1)] It is easy to check that
\begin{align*}
d(x_0, x_1, \ldots, x_{n}) = 0
&\Longleftrightarrow g(x_0, x_1, \ldots, x_{n})=0 \text{ and } \tilde{g}(x_0, x_1, \ldots, x_{n})=0\\
&\Longleftrightarrow x_0 = x_1= \cdots = x_{n}.
\end{align*}
So, $d$ holds the condition $(g1)$. It is clear that $d$ holds the condition $(g2)$.
Let $\{x_i: i=0,\ldots, n\} \subsetneq \{y_i: i=0,\ldots, n\}$. Since $g$ and $\tilde{g}$ are $g$-metrics, it follows that
\begin{align*}
d(x_0, x_1, \ldots, x_{n})
&= g(x_0, x_1, \ldots, x_{n}) + \tilde{g}(x_0, x_1, \ldots, x_{n})\\
&\leq g(y_0, y_1, \ldots, y_{n}) + \tilde{g}(y_0, y_1, \ldots, y_{n}) \\
&=d(y_0, y_1, \ldots, y_{n}).
\end{align*}
Thus $d$ satisfies the condition $(g3)$.
Let $x_0, \ldots, x_s, y_0, \ldots, y_t, w \in \Omega$ with $s + t + 1 = n$.
Then it follows that
\begin{align*}
d(x_0, \ldots, x_s, y_0, \ldots, y_t)
&= g(x_0, \ldots, x_s, y_0, \ldots, y_t) + \tilde{g}(x_0, \ldots, x_s, y_0, \ldots, y_t)\\
&\leq g(x_0, \ldots, x_s, w, \ldots, w) + g(y_0, \ldots, y_t, w, \ldots, w) \\
&\quad+ \tilde{g}(x_0, \ldots, x_s, w, \ldots, w) + \tilde{g}(y_0, \ldots, y_t, w, \ldots, w)\\
& \leq d(x_0, \ldots, x_s, w, \ldots, w) + d(y_0, \ldots, y_t, w, \ldots, w).
\end{align*}
Thus it satisfies the condition $(g4)$.
Therefore, $d$ is a $g$-metric.

\item[(2)]
Since $\psi(x)=0$ if and only if $x=0$, it holds the condition $(g1)$.
It is clear that $d$ holds the condition $(g2)$.
Let $\{x_i: i=0,\ldots, n\} \subsetneq \{y_i: i=0,\ldots, n\}$. Since $g$ is a  $g$-metric, it follows that $g(x_0, x_1, \ldots, x_{n}) \leq g(y_0, y_1, \ldots, y_{n})$. Since $\psi$ is increasing on $[0,\infty)$,
\begin{align*}
d(x_0, x_1, \ldots, x_{n})
&= \psi(g(x_0, x_1, \ldots, x_{n}))\\
&\leq  \psi(g(y_0, y_1, \ldots, y_{n})) \\
&=d(y_0, y_1, \ldots, y_{n}).
\end{align*}
Thus such $d$ satisfies the condition $(g3)$.
Let $x_0, \ldots, x_s, y_0, \ldots, y_t, w \in \Omega$ with $s + t + 1 = n$.
Then it follows that
\begin{align*}
d(x_0, \ldots, x_s, y_0, \ldots, y_t)
&= \psi(g(x_0, \ldots, x_s, y_0, \ldots, y_t)) \\
& \leq \psi(g(x_0, \ldots, x_s, w, \ldots, w) + g(y_0, \ldots, y_t, w, \ldots, w)) \\
& \leq \psi(g(x_0, \ldots, x_s, w, \ldots, w)) + \psi(g(y_0, \ldots, y_t, w, \ldots, w))\\
& =d(x_0, \ldots, x_s, w, \ldots, w) +  d(y_0, \ldots, y_t, w, \ldots, w).
\end{align*}
Thus it satisfies the condition $(g4)$.
Therefore, $d$ is a $g$-metric.
\end{itemize}
\end{proof}

\begin{example}\label{ex:newold:g-metric}
The following functions, denoted by $\psi$,  satisfy the conditions in Lemma \ref{lemma:new_g-metric} (2).
Thus, each $\psi \circ g$ is a $g$-metrics for any $g$-metric $g$.
\begin{itemize}
\item[(1)]  $(\psi \circ g)(x_0, \ldots, x_n) =kg(x_0, \ldots, x_n)$ where $\psi(x)=kx$ with a fixed $k>0$.

\item[(2)]  $(\psi \circ g)(x_0, \ldots, x_n) = \dfrac{g(x_0, \ldots, x_n)}{1+g(x_0, \ldots, x_n)}$ where $\psi(x)=\dfrac{x}{1+x}$.

\item[(3)]  $(\psi \circ g)(x_0, \ldots, x_n) = \sqrt{g(x_0, \ldots, x_n)}$ where $\psi(x)=\sqrt{x}$. Furthermore,  it is true for $\psi(x)=x^{1/p}$ with a fixed $p \geq1$.

\item[(4)] $(\psi \circ g)(x_0, \ldots, x_n) =\log{(g(x_0, \ldots, x_n)+1)}$ where $\psi(x)=\log{(x+1)}$.

\item[(5)] $(\psi \circ g)(x_0, \ldots, x_n) =\min\{ k, g(x_0, \ldots, x_n) \}$ where $\psi(x)=\min\{ k, x \}$ with a fixed $k>0$.
\end{itemize}
\end{example}

We give several interesting examples of $g$-metric on a variety of settings in the following.

\begin{example}
(Discrete $g$-metric) For a nonempty set $\Omega$, define $d: \Omega^{n+1} \rightarrow \R_+$ by
\begin{equation*}
d(x_0, \ldots, x_n)=
\begin{cases}
0 \quad \text{if } x_{0} = \cdots = x_{n}, \\
1 \quad \text{otherwise }
\end{cases}
\end{equation*}
for all $x_0, \ldots, x_n \in \Omega$.
Then $d$ is a $g$-metric on $\Omega$.
\end{example}
\begin{proof}
It is trivial that $d$ satisfies the conditions $(g1)$ and $(g2)$.
\begin{itemize}
\item[$(g3)$] Let $x_0, \ldots,~ x_n,~y_0, \ldots,~ y_n \in \Omega$ such that
$\{x_0, \ldots, x_n \} \subsetneq \{y_0, \ldots, y_n \}$.
If $n(\{ x_0, \ldots, x_n \}) = 1$, then $x_{0} = \cdots = x_{n}$, and so\\ $g(x_0, \ldots, x_n) = 0 < 1 = g(y_0, \ldots, y_n)$. If $n(\{ x_0, \ldots, x_n \}) > 1$, then $g(x_0, \ldots, x_n) = 1 = g(y_0, \ldots, y_n)$.
\item[$(g4)$] Let $x_0, \ldots,~ x_s,~y_0, \ldots,~y_t,~w \in \Omega$ with $s + t + 1 = n$.
If $d(x_0, \ldots, x_s, w, \ldots, w) = 1$ or $d(y_0, \ldots, y_t, w, \ldots, w) = 1$, then
$$ g(x_0, \ldots, x_s, y_0, \ldots, y_t) \leq 1 \leq g(x_0, \ldots, x_s, w, \ldots, w) + g(y_0, \ldots, y_t, w, \ldots, w). $$
If $d(x_0, \ldots, x_s, w, \ldots, w) = 0$ and $d(y_0, \ldots, y_t, w, \ldots, w) = 0$, then\\ $x_i = y_j = w$ for all $i = 0, 1, \dots, s$ and $j = 0, 1, \dots, t$.\\
So $g(x_0, \ldots, x_s, y_0, \ldots, y_t) = 0$. Thus,
$$ g(x_0, \ldots, x_s, y_0, \ldots, y_t) = 0 = g(x_0, \ldots, x_s, w, \ldots, w) + g(y_0, \ldots, y_t, w, \ldots, w). $$
\end{itemize}
\end{proof}

\begin{example} (Diameter $g$-metric)
Define $d: \R_+^{n+1} \longrightarrow \R_+$ by
\begin{equation*}
d(x_0, \ldots, x_n) = \max_{0 \leq i \leq n}{x_{i}} - \min_{0 \leq j \leq n}{x_{j}}
\end{equation*}
for all $x_0, \ldots, x_n \in \R_+$.
Then $d$ is a $g$-metric on $\R_+$.
\end{example}
\begin{proof}
It is easy to check that $d$ holds the conditions $(g1)$, $(g2)$, and $(g3)$. Let us show that such $d$ holds the condition $(g4)$.
Let $x_0, \ldots, x_s, y_0, \ldots, y_t, w \in \R_+$ with $s+t+1=n$.
Set $M_x = \max\{ x_0, \ldots, x_s \}$, $m_x = \min\{ x_0, \ldots, x_s \}$,
$M_y = \max\{ y_0, \ldots, y_t \}$, and $m_y = \min\{ y_0, \ldots, y_t \}$.
Without loss of generality, we assume that $M_x \leq M_y$.
Then there are  three different cases:
(i) $m_x \leq M_x \leq m_y \leq M_y$; (ii) $m_x \leq m_y \leq M_x \leq M_y$; (iii) $m_y \leq m_x \leq M_x \leq M_y$.
For notational simplicity we denote $A= d(x_0, \ldots, x_s,y_0,\ldots,y_t)$,
$B=d(x_0, \ldots, x_s,w,\ldots,w)$, $C=d(y_0, \ldots, y_t,w\ldots,w)$.
For the case (i), clearly $A=M_y - m_x$, and there are five different possibilities with respect to $w$ as follows:
\begin{align*}
m_x \leq M_x \leq m_y \leq M_y \leq w &  \Longrightarrow  A  \leq  w- m_x = B \leq B+C; \\
m_x \leq M_x \leq m_y \leq w \leq M_y &  \Longrightarrow  A =  M_y - m_y + m_y - m_x \leq B+C; \\
m_x \leq M_x \leq w \leq m_y \leq M_y &  \Longrightarrow  A=  M_y - w + w - m_x = B +C; \\
m_x \leq w \leq M_x \leq m_y \leq M_y &  \Longrightarrow  A =  M_y - w + w - m_x \leq  B+C;\\
w \leq m_x \leq M_x \leq m_y \leq M_y  &  \Longrightarrow A  \leq  M_y- w = C \leq B+C.
\end{align*}
For the case (ii), $A  = M_y - m_x $. There are five different possibilities with respect to the value of $w$ as follows:
\begin{align*}
m_x \leq m_y \leq M_x \leq M_y \leq w &  \Longrightarrow A   \leq  w- m_x = B \leq B+C; \\
m_x \leq m_y \leq M_x \leq w \leq M_y &  \Longrightarrow A =   M_y - m_y + m_y - m_x \leq  B +C; \\
m_x \leq m_y \leq w \leq M_x \leq M_y &  \Longrightarrow A =   M_y - w + w - m_x \leq  B+C;\\
m_x \leq w \leq m_y \leq M_x \leq M_y &  \Longrightarrow A = M_y - w + w - m_x \leq B+C;\\
w \leq m_x \leq m_y \leq M_x \leq M_y &  \Longrightarrow A \leq  M_y- w  = C \leq B+C .
\end{align*}
For the case (iii),
when $M_y \leq w$ or $w \leq m_y$, it holds that $d(x_0, \ldots, x_s,y_0,\ldots,y_t)  \leq d(y_0, \ldots, y_t,w\ldots,w)$.
Otherwise, $d(x_0, \ldots, x_s,y_0,\ldots,y_t)  = d(y_0, \ldots, y_t,w\ldots,w)$.
Therefore, $d$ satisfies the condition $(g4)$.
\end{proof}

\begin{remark}
For a nonempty normed space $(\Omega, \| \cdot \|)$ we define $d: \Omega^{n+1} \longrightarrow \R_+$ by
\begin{equation*}
d(x_0, \ldots, x_n) = \max_{0 \leq i \leq n}{\| x_{i} \|} - \min_{0 \leq j \leq n}{\| x_{j} \|}
\end{equation*}
for all $x_0, \ldots, x_n \in \Omega$. Then it is not a $g$-metric on $\Omega$. In fact, it holds $(g2)$, $(g3)$, and $(g4)$, but does not hold $(g1)$ in general. Indeed, there possibly exist $x_0, x_1,\ldots, x_n\in \Omega$ such that $\| x_0\| = \| x_1\| = \cdots = \| x_n\|$ although $x_i \neq x_j$ for some $i\neq j$.
\end{remark}

\begin{example} \label{Ex:average g-metric}
(Average $g$-metric) For a given metric space $(\Omega, \delta)$, define $d: \Omega^{n+1} \longrightarrow \R_+$ by
\begin{equation*}
d(x_0, \ldots, x_n) = \frac{1}{(n+1)^2} \sum_{i,j=0}^{n} \delta (x_{i}, x_{j})
\end{equation*}
for all $x_0, \ldots, x_n \in \Omega$.
Then $d$ is a $g$-metric on $\Omega$.
\end{example}
\begin{proof}
By Example \ref{ex:newold:g-metric} (1), it is enough to show that
$\displaystyle d(x_0, \ldots, x_n) = \sum_{i<j} \delta (x_{i}, x_{j})$ is a $g$-metric on $\Omega$. Clearly, $d$ holds the conditions $(g2)$ and $(g3)$.
\begin{itemize}
    \item[$(g1)$] Since $\delta$ is a metric on $\Omega$, it follows that
$d(x_0, \ldots, x_n) = 0$ if $x_0 = \cdots = x_n$.
Conversely, if $d(x_0, \ldots, x_n) = 0$, then $\delta(x_{i}, x_{j}) = 0$ for all $i,j = 0, \dots, n$. So $x_0 = \cdots = x_n$.
    \item[$(g4)$] Let $x_0, \ldots, x_s, y_0, \ldots, y_t, w \in \Omega$ with $s + t + 1 = n$.
Since $\delta$ is a metric on $\Omega$, it holds from the triangle inequality that
$\delta(x_{i}, y_{j}) \leq \delta(x_{i}, w) + \delta(y_{j}, w)$ for all $i = 0, \dots, s$ and $j = 0, \dots, t$.
Then it follows that
\begin{align*}
    \sum_{i,j} \delta(x_{i}, y_{j}) \leq \sum_{i} \delta(x_{i}, w) + \sum_{j} \delta(y_{j}, w).
\end{align*}
Adding $\displaystyle \sum_{i<j} \delta(x_{i}, x_{j})$ and $\displaystyle \sum_{i<j} \delta(y_{i}, y_{j})$ on both sides, we have
\begin{align*}
d(x_0, \ldots, x_s, y_0, \ldots, y_t) \leq d(x_0, \ldots, x_s, w, \ldots, w) + d(y_0, \ldots, y_t, w, \ldots, w).
\end{align*}
\end{itemize}
\end{proof}

\begin{example} \label{Ex:max g-metric}
(Max $g$-metric) For a given metric space $(\Omega, \delta)$, define $d: \Omega^{n+1} \longrightarrow \R_+$ by
\begin{equation*}
d(x_0, \ldots, x_n) = \max_{0 \leq i,j \leq n} \delta(x_{i}, x_{j})
\end{equation*}
for all $x_0, \ldots, x_n \in \Omega$.
Then $d$ is a $g$-metric on $\Omega$.
\end{example}
\begin{proof}
Obviously, $d$ satisfies $(g1)$, $(g2)$, and $(g3)$.
Let $x_0, \ldots,~ x_s,~y_0, \ldots,~y_t,~w \in \Omega$ with $s + t + 1 = n$.
Let $a$ and $b$ be distinct elements in $\Omega$ such that
$$ d(x_{0}, \dots, x_{s}, y_{0}, \dots, y_{t}) = \delta(a,b).$$
Then there are three different possibilities:
(i) $a, b \in \{ x_0, \ldots, x_s \}$; (ii) $a, b \in \{ y_0, \ldots, y_t \}$; (iii) $a \in \{ x_0, \ldots, x_s \}$, $b \in \{ y_0, \ldots, y_t \}$.
For (i) and (ii), it is clear that $d$ holds $(g4)$.
For (iii), since
$\delta(x_{i}, y_{j}) \leq \delta(x_{i}, w) + \delta(x_{j}, w)$ for all $i = 0, 1, \dots, s$ and $j = 0, 1, \dots, t$, it follows that $\displaystyle \max_{i,j} \delta(x_{i}, y_{j}) \leq \max_{i} \delta(x_{i}, w) + \max_{j} \delta(y_{j}, w)$.
Thus,
\begin{displaymath}
\begin{split}
d(x_0, \ldots, x_s, y_0, \ldots, y_t) & = \delta(a,b) = \max_{i,j} \delta(x_{i}, y_{j}) \\
& \leq \max_{i} \delta(x_{i}, w) + \max_{j} \delta(y_{j}, w) \\
& \leq d(x_0, \ldots, x_s, w, \ldots, w) + d(y_0, \ldots, y_t, w, \ldots, w).
\end{split}
\end{displaymath}
\end{proof}

\begin{remark}
In Example \ref{Ex:average g-metric}, on a given metric space $(\Omega, \delta)$
\begin{equation*}
d(x_0, \ldots, x_n) = \sum_{i,j=0}^{n} \delta (x_{i}, x_{j})
\end{equation*}
is a $g$-metric by Example \ref{ex:newold:g-metric} (1).
Then this $g$-metric and the max $g$-metric in Example \ref{Ex:max g-metric} can be considered as
\begin{displaymath}
\begin{split}
d(x_0, \ldots, x_n) & = \sum_{i,j=0}^{n} \delta (x_{i}, x_{j}) = || M ||_{1}, \\
d(x_0, \ldots, x_n) & = \max_{0 \leq i,j \leq n} \delta(x_{i}, x_{j}) = || M ||_{\infty},
\end{split}
\end{displaymath}
where $M = [m_{ij}]_{0\leq i,j \leq n}$ is the $(n+1) \times (n+1)$ matrix whose entries are $m_{ij} = \delta(x_{i}, x_{j})$. Here, $|| \cdot ||_{1}$ and $|| \cdot ||_{\infty}$ are $\ell_{1}$ and $\ell_{\infty}$ matrix norms, respectively.
So it is a natural question whether or not $|| M ||_{p}$ for $1 < p < \infty$ is a $g$-metric on the metric space $(\Omega, \delta)$.
\end{remark}

\begin{example}\label{Ex:shortest_path}
(Shortest path $g$-metric)
Let $(\Omega, \delta)$ be a nonempty metric space and let $d: \Omega^{3} \to \mathbb{R}_{+}$ be a map defined by
$$ d(x,y,z) = \min\{ \delta(x,y) + \delta(y,z), \delta(x,z) + \delta(z,y), \delta(y,x) + \delta(x,z) \}. $$
Then $d$ is a $g$-metric with order $2$.
\end{example}
\begin{proof}
By Theorem \ref{thm:relation} (2), it is enough to show that $d$ is a $G$-metric.
\begin{itemize}
\item[(G1)] Since $\delta$ is a metric on $\Omega$, $\delta(x,x) = 0$ for all $x \in \Omega$. So $d(x,x,x) = 0$.
\item[(G2)] $d(x,x,y) = \min\{ \delta(x,y), 2\delta(x,y), \delta(x,y) \} = \delta(x,y) > 0$ for all $x \neq y$.
\item[(G3)] $d(x,x,y) = \delta(x,y) \leq d(x,y,z)$ for all $x,y,z \in \Omega$.
\item[(G4)] $d(x,y,z) = d(y,z,x) = d(z,x,y) = d(x,z,y) = d(z, y, x) = d(y,x,z)$.
\item[(G5)] Without loss of generality, assume that $d(w,y,z) = \delta(w,y) + \delta(y,z)$. Then it holds that
\begin{equation*}
d(x,w,w) + d(w,y,z)
= \delta(x,w) + \delta(w,y) + \delta(y,z)
\geq \delta(x,y) + \delta(y,z)
\geq d(x,y,z).
\end{equation*}
The first inequality follows from the triangle inequality for the metric $\delta$, and the second inequality follows from the definition of $d(x,y,z)$.
\end{itemize}
Thus, $d$ is a $G$-metic. Therefore, $d$ is a $g$-metric with order 2.
\end{proof}

\begin{remark}
\begin{itemize}
\item[(1)] All $g$-metrics listed above are multiplicity-independent.

\item[(2)] For a given metric space $(\Omega, \delta)$ we can generalize the $g$-metric with order $2$ in Example \ref{Ex:shortest_path} as a map $d: \Omega^{n+1} \to \mathbb{R}_{+}$ defined by
\begin{equation*}
d(x_0, \ldots, x_n) = \min_{ \pi \in \mathcal{S}} \sum_{i=0}^{n-1} \delta( x_{\pi(i)}, x_{\pi(i+1)} )
\end{equation*}
for all $x_0, \ldots, x_n \in \Omega$.
Here, $\mathcal{S}$ denotes the set of all permutations on $\{ 0, 1, \ldots, n \}$.
That is, $d(x_0, \ldots, x_n)$ is the length of the shortest path connecting $x_0, \ldots, x_n$.
Finding the shortest path is very important problem in operations research and theoretical computer science, which is also known as the traveling salesman problem\cite{PR91,Verblunsky1951}.
In Example \ref{Ex:shortest_path} we showed that the shortest path $g$-metric is a $g$-metric with order $2$, but it is an open problem that $d$ is a $g$-metric for any $n \geq 3$.

\item[(3)]  Let $\Omega$ be a nonempty subset of $\mathbb{R}^{n},$ i.e., $\Omega$ can be considered as an $n$-dimensional data set. Define $d:\Omega^{n+1} \longrightarrow \mathbb{R}_{+}$ by
$d(x_{0},\ldots,x_{n})  $ is the diameter of the smallest closed ball, $B$, such that $\{x_{0},\ldots,x_{k}\} \subseteq B$.
This is called the smallest enclosing circle problem, which was introduced by Sylvester\cite{sylvester1857}.
For more information, see \cite{DRAGER2007929,Mordukhovich2013,Welzl91}.
It is an open problem that $d$ is a $g$-metric for any $n \geq 3$.
\end{itemize}
\end{remark}

\begin{theorem}\label{thm:g-metric:basic-properties}
Let $g$ be a $g$-metric with order $n$ on a nonempty set $\Omega$.
The following are true:
\begin{itemize}
    \item[(1)] $g(\underbrace{x, \ldots, x}_\text{$s$ times},y, \ldots, y) \leq g( \underbrace{x, \ldots, x}_\text{$s$ times}, w, \ldots, w) + g(\underbrace{w, \ldots, w}_\text{$s$ times},y, \ldots, y)$,
    \item[(2)] $g(x, y,  \ldots, y) \leq g(x, w,  \ldots, w) + g(w, y,  \ldots, y)$, and \\
$ g(x, y, \ldots, y) \leq g(x, w, \ldots, w) + g(y, w, \ldots, w) $  if $g$ is multiplicity-independent,
    \item[(3)] $g(\underbrace{x,  \ldots, x}_\text{$s$ times},w,\ldots, w) \leq s g(x, w, \ldots, w)$ and \\
    $g(\underbrace{x,  \ldots, x}_\text{$s$ times},w,\ldots, w) \leq (n+1-s) g(w, x, \ldots, x)$,
    \item[(4)] $ \displaystyle{g(x_0, x_1, \ldots, x_n) \leq \sum_{i=0}^{n} g(x_{i}, w, \ldots, w)}$,
    \item[(5)] $\big| g(y, x_1, \ldots, x_n) - g(w, x_1, \ldots, x_n) \big| \leq \max\{ g(y, w,  \ldots, w), g(w, y,  \ldots, y) \}$,
     \item[(6)] $\big| g(\underbrace{x, \ldots, x}_\text{$s$ times},w,\ldots,w) - g(\underbrace{x, \ldots, x}_\text{$\tilde{s}$ times},w,\ldots,w) \big| \leq \big| s-\tilde{s} \big| g(x, w,  \ldots, w) $.
     \item[(7)]    $g(x, w, \ldots, w) \leq (1+(s-1)(n+1-s))g(\underbrace{x,  \ldots, x}_\text{$s$ times},w,\ldots, w) $,
  \end{itemize}

\end{theorem}
\begin{proof}
(1) and (2) follow from the condition $(g4)$.
Note that for a multiplicity-independent $g$-metric $g$, it is true that $g(y, w,  \ldots, w) = g(w, y,  \ldots, y)$.
\begin{itemize}
\item[(3)] By the condition $(g4)$, it follows that
\begin{align*}
g(\underbrace{x,  \ldots, x}_\text{$s$ times},w,\ldots, w)
&\leq g(\underbrace{x,  \ldots, x}_\text{$s-1$ times}, w, w) + g(x, w, \ldots, w) \\
&\leq g(\underbrace{x,  \ldots, x}_\text{$s-2$ times}, w, w, w) + g(x, w,  \ldots, w) + g(x, w,  \ldots, w) \\
& \quad \vdots \\
&\leq g(x, w,  \ldots, w) + g(x, w,  \ldots, w) + \cdots + g(x, w,  \ldots, w) \\
&\leq s g(x, w, \ldots, w).
\end{align*}

\item[(4)] By the condition $(g2)$ and $(g4)$, it follows that
\begin{align*}
g(x_0, x_1, \ldots, x_n)
&\leq g(x_0, w,  \ldots, w) + g(x_1, x_2, \ldots, x_n, w) \\
&\leq g(x_0, w,  \ldots, w) + g(x_1, w, \ldots, w) + g(x_2, \ldots, x_n, w, w) \\
& \quad \vdots \\
&\leq \sum_{i=0}^{n} g(x_{i}, w, \ldots, w).
\end{align*}
\item[(5)] By the condition $(g4)$, it follows that
$ g(y, x_1, \ldots, x_n) \leq g(w, x_1, \ldots, x_n) + g(y, w, \ldots, w)$.
Then
$$ g(y, x_1, \ldots, x_n) - g(w, x_1, \ldots, x_n) \leq g(y, w, \ldots, w). $$
Similarly, we have
$$ g(w, x_1, \ldots, x_n) - g(y, x_1, \ldots, x_n) \leq g(w, y, \ldots, y). $$

\item[(6)] By (3), it is trivial.

\item[(7)] By Theorem \ref{thm:g-metric:basic-properties} (3), we have
\begin{align*}
g(x, w, \ldots, w)
& \leq g(x,x,w,\ldots,w) + g(w,x,\ldots,x)\\
& \leq g(x,x,x,w,\ldots,w) + g(w,x,\ldots,x)+g(w,x,\ldots,x)\\
& \quad \vdots \\
&\leq g(\underbrace{x,\ldots,x}_\text{$s$ times}, w,  \ldots, w) + (s-1)g(w,x,\ldots,x)  \\
&\leq g(\underbrace{x,\ldots,x}_\text{$s$ times}, w,  \ldots, w) + (s-1)(n+1-s)g(\underbrace{x,\ldots,x}_\text{$s$ times}, w,  \ldots, w)\\
&= (1+(s-1)(n+1-s))g(\underbrace{x,\ldots,x}_\text{$s$ times}, w,  \ldots, w).
\end{align*}

\end{itemize}
\end{proof}

For a given $g$-metric, we can construct a distance function.
\begin{example}
For any $g$-metric space $(\Omega,g)$, the following are distance functions:
\begin{itemize}
\item[(1)] $d(x,y)= g(\underbrace{x, \ldots, x}_\text{$s$ times},y, \ldots, y) +  g(\underbrace{y, \ldots, y}_\text{$s$ times}x, \ldots, x)$,
\item[(2)] $d(x,y) = g(x,y,\ldots, y) + g(x,x,y,\ldots, y) + \cdots + g(x,x,\ldots, x,y)$,
\item[(3)] $d(x,y) = \max\{ g(x_0, x_1, \ldots, x_n) : x_{i} \in \{ x,y \},~ 0 \leq i \leq n \}$.
\end{itemize}
\end{example}

\begin{proof}
It is easy to show that each function $d$ holds the conditions (1), (2), and (3) in Definition \ref{D:distance}.
We show that each function holds the triangle inequality (4) in Definition \ref{D:distance}.
\begin{itemize}
   \item[(1)]
By the condition $(g4)$, it follows that
\begin{align*}
d(x,z) + d(z,y)
&= g(\underbrace{x, \ldots, x}_\text{$s$ times},z, \ldots, z) +  g(\underbrace{z, \ldots, z}_\text{$s$ times},x, \ldots, x) \\
&\quad + g(\underbrace{z, \ldots, z}_\text{$s$ times},y, \ldots, y) +  g(\underbrace{y, \ldots, y}_\text{$s$ times},z, \ldots, z)\\
& \geq
g(\underbrace{x, \ldots, x}_\text{$s$ times},y, \ldots, y) +  g(\underbrace{y, \ldots, y}_\text{$s$ times},x, \ldots, x) \\
&= d(x,y).
\end{align*}

    \item[(2)] By the condition $(g4)$, it follows that
\begin{align*}
d(x,z) + d(z,y)
&= g(x,z,\ldots, z)+g(x,x,z,\ldots, z) + \cdots + g(x,x,\ldots, x,z)\\
&\quad +g(z,y,\ldots, y)+g(z,z,y,\ldots, y) + \cdots + g(z,z,\ldots, z,y)\\
&\geq g(x,y,\ldots, y)+g(x,x,y,\ldots, y) + \cdots + g(x,x,\ldots, x,y)\\
&= d(x,y).
\end{align*}

    \item[(3)] Let $x, y \in \Omega$. Since $d(x,y)=0$ for $x=y$, the function $d$ holds the triangle inequality. If $x\neq y$, then there exists $\alpha_0, \alpha_1, \ldots, \alpha_n \in \{ x,y \}$ such that $d(x,y)=g(\alpha_0, \alpha_1, \ldots, \alpha_n)$ and $$g(\alpha_0, \alpha_1, \ldots, \alpha_n) \geq g(\beta_0, \beta_1, \ldots, \beta_n)$$ for all $\beta_0, \beta_1, \ldots, \beta_n \in  \{ x,y \} $.
Let $A = \{i: \alpha_i =x \}$. Clearly, $1\leq n(A) \leq n$. Take $d= n(A)$.
Without loss of generality,  we assume that $\alpha_0 = \cdots = \alpha_d = x$ and $\alpha_{d+1} = \cdots = \alpha_{n} = y$.
That is, $d(x,y) = g(\underbrace{x,x,\dots,x}_\text{$d$ times}, y,y, \ldots,y)$.
 Then by the condition $(g4)$ it follows that
 \begin{align*}
 d(x,z) + d(z,y)
 & \geq g(\underbrace{x,x,\dots,x}_\text{$d$ times},z,z,\ldots,z) + g(\underbrace{z,z,\ldots, z}_\text{$d$ times}, y,y,\ldots,y)\\
 & \geq g(\underbrace{x,x,\dots,x}_\text{$d$ times},y,y,\ldots,y) = d(x,y).
 \end{align*}

\end{itemize}
\end{proof}

\section{Topology on a $g$-metric space}

For a given metric space $(\Omega,d)$, we denote the ball centered at $x_0$ with radius $r$ by $B_{d}(x_0,r).$
We define a ball on a $g$-metric space.
\begin{definition}
Let $(\Omega,g)$ be a $g$-metric space. For
$x_0 \in \Omega$ and $r>0$, the $g$-ball centered at $x_0$ with radius $r$ is
\begin{equation*}
B_g(x_0,r) = \{y\in \Omega: g(x_0, y, y, \ldots, y)<r \}.
\end{equation*}
\end{definition}

\begin{proposition}
Let $(\Omega,g)$ be a $g$-metric space. Then the following hold.
\begin{itemize}
\item[(1)] If $g(x_0,x_1,x_2,\ldots,x_n)<r$ and $n(\{x_0,x_1,x_2,\ldots,x_n\}) \geq 3$, then $x_i \in B_g(x_0,r)$ for all $i=0,\ldots,n$.
\item[(2)] If $g$ is multiplicity-independent and $g(x_0,x_1,x_2,\ldots,x_n)<r,$ then $x_i \in B_g(x_0,r)$ for all $i=0,\ldots,n$.
\item[(3)] Let $y \in B_g(x_1,r_1) \cap B_g(x_2,r_2).$ Then there exists $\delta > 0$ such that $B_g(y, \delta) \subseteq B_g(x_1,r_1) \cap B_g(x_2,r_2).$
\end{itemize}
\end{proposition}
\begin{proof}
Suppose that $g(x_0,x_1,x_2,\ldots,x_n)<r$. Set $X=\{ x_0,x_1,x_2,\ldots,x_n \}$.
\begin{itemize}
\item[(1)] Since $n(X) \geq 3$, clearly $\{ x_0, x_i, x_i, \ldots, x_i \} \subsetneq X$ for each $i\in \N$. By monotonicity condition for the $g$-metric, it follows that
$g(x_0, x_i, \ldots, x_i) \leq g(x_0,x_1,\ldots,x_n) <r$. So $x_i \in B_g(x_0,r)$ for all $i\in \N$.
\item[(2)] It suffices to show that it holds for $n(X) = 2$. Since a $g$-metric is multiplicity-independent,
$g(x_0, x_i,  \ldots, x_i) \leq g(x_0,x_1,\ldots,x_n) <r$.
\item[(3)] Since $y \in B_g(x_1,r_1) \cap B_g(x_2,r_2),$ it holds that $g(x_i, y, \ldots, y)<r_i$ for $i=1,2$. We take $\delta = \min\{r_i - g(x_i, y, \ldots, y):i=1,2\}.$ Then for every $z \in B_g(y, \delta),$ by Theorem \ref{thm:g-metric:basic-properties} (2) we have $g(x_i, z, \ldots, z) \leq g(x_i, y, \ldots, y) + g(y, z, \ldots, z) < g(x_i, y, \ldots, y) + \delta < r_i$ for each $i=1,2.$ Therefore, $B_g(y, \delta) \subseteq B_g(x_1,r_1) \cap B_g(x_2,r_2).$
\end{itemize}
\end{proof}

Due to the preceding proposition, the collection of all $g$-balls, $\mathcal{B}=\{B_g(x,r) : x\in  \Omega, r>0  \}$ forms a basis for a topology on $\Omega.$ We call the topology generated by $\mathcal{B}$ the \emph{$g$-metric topology} on $\Omega.$

\begin{theorem}
Let $(\Omega,g)$ be a $g$-metric space and let $d(x,y)=g(x,y,\ldots,y) + g(y,x,\ldots,x)$. Then
\begin{equation*}
   B_g \Big(x_0,\frac{r}{n+1} \Big) \subseteq B_{d}(x_0,r) \subseteq B_g(x_0,r).
\end{equation*}
\end{theorem}

\begin{proof}
Recall that $y\in B_g(x_0,r) \Longleftrightarrow g(x_0,y,y,\ldots,y) < r$.

(i) Let $\displaystyle{x \in  B_g\Big(x_0,\frac{r}{n+1} \Big)} $. Then $\displaystyle{g(x_0,x,x,\ldots,x) < \frac{r}{n+1}}$.
It follows that
\begin{align*}
d(x_0,x)
&= g(x_0,x,x,\ldots,x) + g(x, x_0,x_0, \ldots, x_0)\\
&\leq g(x_0,x,x,\ldots,x) + ng(x_0, x,x, \ldots, x)\\
&\leq (n+1)g(x_0,x,x,\ldots,x) < r.
\end{align*}
So, $x \in B_{d}(x_0,r)$.

(ii) Let $x \in  B_{d}(x_0,r) $. Then $d(x_0,x) = g(x_0,x,x,\ldots,x) +  g(x,x_0,x_0,\ldots,x_0)< r$.
Since $g(x_0,x,x,\ldots,x) \leq n g(x,x_0,x_0,\ldots,x_0) $,
it follows that
$$ \frac{n+1}{n} g(x_0,x,x,\ldots,x) \leq  g(x_0,x,x,\ldots,x) +  g(x,x_0,x_0,\ldots,x_0) < r.$$
Thus, $g(x_0,x,x,\ldots,x) < r,$ i.e., $x \in B_g(x_0,r)$ as desired.
\end{proof}

Thus, every $g$-metric space is topologically equivalent to a metric space arising from the metric $d$. This makes it possible to transport many concepts and results from metric spaces into the $g$-metric setting.

\section{Convergence and continuity in $g$-metric spaces}

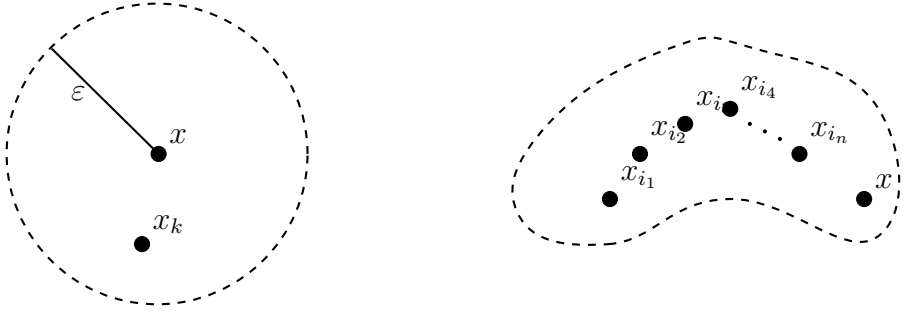
\begin{figure}
\begin{tikzpicture}[scale=2]
\centering
\foreach \Point/\PointLabel in {(2.01,2)/x,(1.9,1.4)/x_k}
\draw[fill=black] \Point circle (0.05) node[above right] {$\PointLabel$};
\draw[thick] (1.3,2.7) --(2.01,2) node[near start,below]{$\varepsilon$};
\draw[thick,dashed] (2,2) circle (1cm);
             \draw[thick,dashed] plot [smooth cycle, tension=0.8] coordinates {(4.4,1.6) (5,1.4) (5.8,1.7) (6.7773,1.4421)(6.7905,2.3074)  (5.9752,2.7) (5.4,2.7) (4.6,2.2) }; 
                          \foreach \Point/\PointLabel in {(5.01,1.7)/x_{i_1},(5.21,2)/x_{i_2},(5.51,2.2)/x_{i_3},(5.81,2.3)/x_{i_4},(6.27,2)/x_{i_n},(6.7,1.7)/x}
\draw[fill=black] \Point circle (0.05) node[above right] {$\PointLabel$};
\draw[fill=black] (5.94,2.2) circle (0.01);
\draw[fill=black] (6.04,2.15) circle (0.01);
\draw[fill=black] (6.14,2.1) circle (0.01);
               \end{tikzpicture}
               \caption{While a convergent sequence is defined by the distance between $x_k$ and $x$ (left), a $g$-convergent sequence is defined by the distance (i.e., $g$-metric) between $x_{i_1},\ldots, x_{i_n}$ and $x$ (right).}\label{fig:ball}
\end{figure}

\begin{definition}\label{def:convergence}
Let $(\Omega,g)$ be a $g$-metric space. Let $x\in \Omega$ be a point and $\{x_k\} \subseteq \Omega$ be a sequence.
\begin{itemize}
    \item[(1)]   $\{x_k\}$ \emph{$g$-converges} to $x$, denoted by $\{x_k\} \overset{g}{\longrightarrow} x$, if for all $\varepsilon >0$ there exists $N\in \N$ such that
    $$i_1,\ldots, i_n \geq N \Longrightarrow g(x,x_{i_1},\ldots, x_{i_{n}})< \varepsilon.$$
  For such a case, $\{x_k\}$ is said to be \emph{$g$-convergent} in $\Omega$ and $x$ is called the \emph{$g$-limit} of $\{x_k\}$. 

    \item[(2)] $\{x_k\}$ is said to be \emph{$g$-Cauchy} if for all $\varepsilon >0$ there exists $N\in \N$ such that
    $$i_0,i_1,\ldots, i_n \geq N \Longrightarrow g(x_{i_0},x_{i_1},\ldots, x_{i_{n}})< \varepsilon.$$

    \item[(3)] $(\Omega,g)$ is \emph{complete} if every $g$-Cauchy sequence in $(\Omega,g)$ is $g$-convergent in $(\Omega,g)$.
\end{itemize}
\end{definition}

\begin{proposition}
The following are true.
\begin{itemize}
    \item[(1)]  The limit of a $g$-convergent sequence in a $g$-metric space is unique.

    \item[(2)] Every convergent sequence in a $g$-metric space is a $g$-Cauchy sequence.
\end{itemize}
\end{proposition}
\begin{proof}
\begin{itemize}
    \item[(1)] Let $(\Omega,g)$ be a $g$-metric space and let $\{x_k\} \subseteq \Omega$ be a $g$-convergent sequence.
    Suppose that $x,y \in \Omega$ are the $g$-limits of $\{x_k\}$. By Definition \ref{def:convergence} (1), there exists $N_1,N_2\in \N$ such that
    \begin{align*}
        g(x,x_{i_1},\ldots, x_{i_{n}}) &< \frac{\varepsilon}{n+1} \quad \text{for all }i_1,\ldots, i_n \geq N_1,\\
         g(y,x_{i_1},\ldots, x_{i_{n}}) &< \frac{\varepsilon}{n+1} \quad \text{for all }i_1,\ldots, i_n \geq N_2.
    \end{align*}
    Set $N=\max\{N_1,N_2\}$. If $m \geq N$, then by the condition $(g4)$ and Theorem \ref{thm:g-metric:basic-properties} (3), it follows that
    \begin{align*}
        g(x,y,y,\ldots,y) &\leq g(x,x_m, x_m, \ldots, x_m) + g(x_m, y, y, \ldots, y)\\
        &\leq g(x,x_m, x_m, \ldots, x_m) + ng(y, x_m, x_m, \ldots, x_m)\\
        &< \frac{\varepsilon}{n+1} + \frac{n\varepsilon}{n+1} = \varepsilon.
    \end{align*}
    Since $\varepsilon$ is arbitrary, $g(x,y,y,\ldots,y)=0$. Thus, $x=y$ by the condition $(g1)$.

    \item[(2)]
    Let $(\Omega,g)$ be a $g$-metric space and let $\{x_k\} \subseteq \Omega$ be a convergent sequence with the $g$-limit $x$. By Definition \ref{def:convergence} (1), there exists $N\in \N$ such that
    \begin{align*}
         g(x,x_{i_1},\ldots, x_{i_{n}}) &< \frac{\varepsilon}{n+1} \quad \text{for all }i_1,\ldots, i_n \geq N.
    \end{align*}
    By Theorem \ref{thm:g-metric:basic-properties} (4) and the monotonicity condition for the $g$-metric, it follows that
    \begin{align*}
        g(x_{i_0},x_{i_1},\ldots, x_{i_{n}}) \leq \sum_{k=0}^n g(x_{i_k},x,x,\ldots, x) <  \sum_{k=0}^n \frac{\varepsilon}{n+1} = \varepsilon.
    \end{align*}
    Thus, $\{x_k\}$ is a $g$-Cauchy sequence in $(\Omega, g)$.
\end{itemize}
\end{proof}

\begin{lemma}\label{lemma:g-convergent}
Let $(\Omega,g)$ be a $g$-metric space. Let $\{x_k\} \subseteq \Omega$ be a sequence and $x\in \Omega$.
The following are equivalent.
\begin{itemize}
\item[(1)]  $\{x_k\} \overset{g}{\longrightarrow} x$.
\item[(2)] For a given $\varepsilon>0 $, there exists $N \in \N$ such that $x_k \in B_g(x,\varepsilon)$ for all $k \geq N$.
\item[(3)] $\displaystyle{\lim_{k_1,\ldots, k_s \rightarrow \infty} g(\underbrace{x_{k_1}, \ldots, x_{k_s}}_\text{$s$ times},x,\ldots,x) =0}$ for a fixed $1 \leq s \leq n$. That is, for all $\varepsilon>0$, there exists $N \in \N$ such that $k_1,\ldots, k_s  \geq N$ $\Longrightarrow$ $g(x_{k_1},\ldots, x_{k_s},x,\ldots,x) < \varepsilon$.
\end{itemize}
\end{lemma}

\begin{proof}
($(1) \Longleftrightarrow (2)$) It is clear by the definition of $g$-convergence.

($(2) \Longrightarrow (3)$) Assume that for a given $\varepsilon > 0,$ there exists $N \in \N$ such that $k \geq N$ implies $\displaystyle{x_{k} \in B_{g} \Big( x, \frac{\varepsilon}{s} \Big)},$ i.e., $\displaystyle{g(x, x_{k}, \ldots, x_{k}) < \frac{\varepsilon}{s}}$.
If $k_1,\ldots, k_s  \geq N$, then by Theorem \ref{thm:g-metric:basic-properties} (4), we have that
$\displaystyle{g(x_{k_1},\ldots, x_{k_s},x,\ldots,x)  \leq \sum_{j=1}^s g(x,x_{k_j},\ldots, x_{k_j}) < \varepsilon }.$

($(3) \Longrightarrow (2)$) Let $\varepsilon>0$. Assume that there exists $N \in \N$ such that
$$k_1,\ldots, k_s  \geq N \Longrightarrow g(k_1,\ldots, k_s,x,\ldots,x) < \frac{\varepsilon}{(1+(s-1)(n+1-s))}.$$
If $k \geq N$, then by Theorem \ref{thm:g-metric:basic-properties} (7) it follows that
\begin{align*}
g(x,x_k,\ldots, x_k)
 \leq (1+(s-1)(n+1-s)) g(\underbrace{x_k,\ldots, x_k}_\text{$s$ times},x,\ldots,x) < \varepsilon.
\end{align*}

\end{proof}


\begin{lemma}\label{lemma:g-cauchy}
Let $(\Omega,g)$ be a $g$-metric space. Let $\{x_k\} \subseteq \Omega$ be a sequence.
The following are equivalent.
\begin{itemize}
\item[(1)] $\{x_k\}$ is $g$-Cauchy.
\item[(2)] $g(x_k, x_{k+1}, x_{k+1},\ldots, x_{k+1}) \longrightarrow 0$ as $k \longrightarrow \infty.$
\item[(3)] $\displaystyle{\lim_{k,\ell \rightarrow \infty} g(\underbrace{x_k,\ldots,x_k}_\text{$s$ times},x_\ell,\ldots,x_\ell) =0}$ for a fixed $1 \leq s \leq n $.
\end{itemize}
\end{lemma}

\begin{proof}
($(1) \Longrightarrow (2)$) It is trivial by Definition \ref{def:convergence} (2).

($(2) \Longrightarrow (3)$) Without loss of generality, we can assume $k < \ell.$ Let $\varepsilon > 0$ be given. Then for each $m=0, \ldots, \ell-k-1$ there exists $N_{m} \in \N$ such that $\displaystyle{g(x_{k+m},x_{k+m+1}, \ldots, x_{k+m+1}) < \frac{\varepsilon}{n(\ell-k)} }$.  Let $N=\text{max}\{N_{0},\ldots, N_{\ell-k-1}\}.$ Then by Theorem \ref{thm:g-metric:basic-properties} (3),(4), and the conditions $(g4)$,  we have that
\begin{align*}
g(\underbrace{x_k,\ldots,x_k}_\text{$s$ times}, & x_\ell,\ldots,x_\ell)
 \leq s g(x_{k},x_{\ell},\ldots,x_{\ell}) \\
&\leq s \big(g(x_{k},x_{k+1},\ldots,x_{k+1}) + g(x_{k+1},x_{\ell},\ldots,x_{\ell}) \big) \\
&\leq s \big(g(x_{k},x_{k+1},\ldots,x_{k+1}) + g(x_{k+1},x_{k+2},\ldots,x_{k+2})+ g(x_{k+2},x_{\ell},\ldots,x_{\ell}) \big) \\
& \quad \vdots \\
&\leq s\sum\limits_{i=k}^{\ell-1} g(x_{i},x_{i+1},\ldots,x_{i+1})< \varepsilon,
 \end{align*}
for all $k \geq N$.
If $k,\ell \geq N$, then $g(\underbrace{x_k,\ldots,x_k}_\text{$s$ times}x_\ell,\ldots,x_\ell)< \varepsilon $.

($(3) \Longrightarrow (1)$)
Let $\varepsilon > 0$ be given.
Assume that  there exists $N \in\N$ such that
$$k,\ell \geq N ~ \Longrightarrow ~ g(\underbrace{x_{k},\ldots,x_{k}}_\text{$s$ times},x_{\ell},\ldots,x_{\ell}) < \frac{\varepsilon}{n(1+(s+1)(n+1-s))}.$$
If $i_0,i_1,\ldots,i_n \geq N$, then by Theorem \ref{thm:g-metric:basic-properties} (4),(7) it follows that
\begin{align*}
g(x_{i_0},x_{i_1},\ldots, x_{i_{n}})
&\leq  \sum_{k=0}^n g(x_{i_k},x_{i_0},\ldots, x_{i_{0}}) \\
&\leq  \sum_{k=0}^n (1+(s+1)(n+1-s))  g(\underbrace{x_{i_k},\ldots,x_{i_k}}_\text{$s$ times},x_{i_0},\ldots,x_{i_0})
< \varepsilon.
 \end{align*}
 \end{proof}

\begin{definition}
Let $(\Omega,g)$ be a $g$-metric space, and let $\varepsilon>0$ be given.
\begin{itemize}
\item[(1)] A set $A\subseteq \Omega$ is called an \emph{$\varepsilon,g$-net} of $(\Omega,g)$ if for each $x\in \Omega$, there exists $a \in A$ such that $x \in B_g(a,\varepsilon)$.
If the set $A$ is finite then $A$ is called a \emph{finite $\varepsilon,g$-net} of $(\Omega,g)$.
\item[(2)] A $g$-metric space  $(\Omega,g)$ is called \emph{totally $g$-bounded} if for every $\varepsilon>0$ there exists a finite $\varepsilon,g$-net.
\item[(3)] A $g$-metric space  $(\Omega,g)$ is called \emph{$g$-compact} if it is complete and totally $g$-bounded.

\end{itemize}

\end{definition}

\begin{definition}
Let $(\Omega_{1},g_{1})$ and $(\Omega_{2},g_{2})$ be $g$-metric spaces.
\begin{itemize}
  \item[(1)] A mapping $T: \Omega_{1} \longrightarrow \Omega_{2}$ is said to be $g$-\emph{continuous at a point} $x \in \Omega_{1}$ provided that for each open ball $B_{g_{2}}(T(x), \varepsilon),$ there exists an open ball $B_{g_{1}}(x, \delta)$ such that $T(B_{g_{1}}(x, \delta)) \subseteq B_{g_{2}}(T(x), \varepsilon).$
  \item[(2)] $T: \Omega_{1} \longrightarrow \Omega_{2}$ is said to be $g$-\emph{continuous} if it is continuous at every point of $\Omega_{1}.$
  \item[(3)] $T: \Omega_{1} \longrightarrow \Omega_{2}$ is called a $g$-\emph{homeomorphism} if $T$ is bijective, and $T$ and $T^{-1}$ are $g$-continuous. In this case, the spaces $\Omega_{1}$ and $\Omega_{2}$ are said to be $g$-\emph{homeomorphic}.
  \item[(4)] A property $P$ of $g$-metric spaces is called a $g$-\emph{topological invariant} if $P$ satisfies the condition:\\
  If a space $\Omega_{1}$ has the property $P$ and if $\Omega_{1}$ and $\Omega_{2}$ are $g$-homeomorphic, then $\Omega_{2}$ also has the property $P.$
\end{itemize}
\end{definition}

\begin{proposition}
Let $(\Omega_{1},g_{1})$ and $(\Omega_{2},g_{2})$ be $g$-metric spaces, and let $T: \Omega_{1} \longrightarrow \Omega_{2}$ be a mapping. Then the following are equivalent.
\begin{itemize}
  \item[(1)] T is $g$-continuous.
  \item[(2)] For each point $x \in \Omega_{1}$ and for each sequence $\{x_{k}\}$ in $\Omega_{1}$ $g$-converging to $x,$ $\{T(x_{k})\}$ $g$-converges to $T(x).$
\end{itemize}
\end{proposition}

\begin{proof}
($(1) \Longrightarrow (2)$) Let $x \in \Omega_{1},$ and let $\{x_{k}\}$ be a sequence in $\Omega_{1}$ $g$-converging to $x.$ Since $T: \Omega_{1} \longrightarrow \Omega_{2}$ is $g$-continuous, for a given $\varepsilon >0$ there exists $\delta >0$ such that $T(B_{g_{1}}(x, \delta)) \subseteq B_{g_{2}}(T(x), \varepsilon n^{-2}).$ Since $\{x_{k}\} \overset{g}{\longrightarrow} x,$ there is $N \in \mathbb{N}$ such that $g(x, x_{i_{1}}, \ldots, x_{i_{n}}) < \delta$ for all $i_{1}, \ldots, i_{n} \geq N.$ Thus $g(x, x_{i_{k}}, \ldots, x_{i_{k}}) < \delta$ for each $k= 1, \ldots, n$.
Then the $g$-continuity of $T$ gives rise to the inequality
$$g(T(x), T(x_{i_{k}}), \ldots, T(x_{i_{k}})) < \frac{\varepsilon}{n^{2}}$$ for each $k\in \N.$
By Theorem \ref{thm:g-metric:basic-properties} (3) and (4) we have
\begin{align*}
g(T(x), T(x_{i_{1}}), \ldots, T(x_{i_{n}}))
& \leq \sum\limits_{k=1}^{n} g(T(x_{i_{k}}), T(x), \ldots, T(x))\\
& \leq \sum\limits_{k=1}^{n} n g(T(x), T(x_{i_{k}}), \ldots, T(x_{i_{k}})) < \varepsilon.
\end{align*}
Therefore, $\{T(x_{k})\}$ $g$-converges to $T(x).$

($(2) \Longrightarrow (1)$) Suppose that $T$ is not $g$-continuous, i.e. there exists $x \in \Omega_{1}$ such that $T$ is not $g$-continuous at $x.$ Then there exists $\varepsilon > 0$ such that for each $\delta > 0$ there is $y \in \Omega_{1}$ with $g(x, y, \ldots, y) < \delta$ but $g(T(x), T(y), \ldots, T(y)) \geq \varepsilon.$ Then for each $k \in \N$ we can take $x_{k} \in \Omega_{1}$ such that $g(x, x_{k}, \ldots, x_{k}) < \frac{1}{k}$ but $g(T(x), T(x_{k}), \ldots, T(x_{k})) \geq \varepsilon.$ Hence, $\{x_{k}\}$ $g$-converges to $x$ but $\{T(x_{k})\}$ does not $g$-converges to $T(x),$ which contradicts to (2).
\end{proof}


A $g$-metric space $(\Omega,g)$ is said to \emph{have the fixed point property} if every $g$-continuous mapping $T: \Omega \longrightarrow \Omega$ has a fixed point.

\begin{proposition}
The fixed point property is a $g$-topological invariant.
\end{proposition}

\begin{proof}
Let $(\Omega_{1},g_1)$ and $(\Omega_{2},g_2)$ be $g$-metric spaces, and let $h:\Omega_{1} \longrightarrow \Omega_{2}$ be a $g$-homeomorphism. Suppose that $\Omega_{1}$ has the fixed point property.\\
Let $\widetilde{T}:\Omega_{2} \longrightarrow \Omega_{2}$ be a $g$-continuous function. We consider the function $T:\Omega_{1} \longrightarrow \Omega_{1}$ given by $T(x) = (h^{-1} \circ \widetilde{T} \circ h)(x).$ Since $\Omega_{1}$ has the fixed point property and $T$ is $g$-continuous, there exists a fixed point $x \in \Omega_{1}$ under $T,$ i.e. $T(x)=x.$ Denote $h(x)$ by $y.$ Then we have
$$\widetilde{T}(y) = \widetilde{T}(h(x)) = (h \circ h^{-1} \circ \widetilde{T} \circ h)(x) = h(T(x)) = h(x) = y,$$
implying that $y$ is a fixed point under $\widetilde{T}.$ Therefore, $\Omega_{2}$ has the fixed point property.
\end{proof}

\begin{lemma}\label{lemma:g-continuous}
If $(\Omega,g)$ is a $g$-metric space, then the function $g$ is jointly continuous in all $n+1$ variables, i.e.,
if for each $i=0,1,\ldots, n$, $\{x_i^{(k)}\}_{k\in\N}$ is a sequence in $\Omega$ such that
$\{x_i^{(k)}\} \overset{g}{\longrightarrow} x_i$, then
$\{ g(x_0^{(k)},x_1^{(k)},\ldots, x_n^{(k)}) \}{\longrightarrow} \{ g(x_0,x_1,\ldots, x_n) \}$ as $k \longrightarrow \infty$.
\end{lemma}

\begin{proof}
Assume that $\{x_{i}^{(k)}\} \overset{g}{\longrightarrow} x_{i}$ as $k \longrightarrow \infty$ for each $i=0, \ldots, n$. For a given $\varepsilon > 0,$ there exists $N_{i} \in \N$ such that
 $\displaystyle{g(x_{i}^{(k)},x_{i}, \ldots, x_{i}) < \frac{\varepsilon}{n+1}}$
  if $k \geq N_{i}$ by Lemma \ref{lemma:g-convergent} (3). We let $N=\text{max}\{N_{0}, N_{1}, \ldots, N_{n}\}.$ Then by the conditions $(g2),(g4)$,  if $k \geq N$, then
\begin{align*}
g(x_0^{(k)},x_1^{(k)},\ldots, x_n^{(k)})
&\leq g(x_0^{(k)},x_0,\ldots, x_0) + g(x_0,x_1^{(k)},\ldots, x_n^{(k)}) \\
&\leq g(x_0^{(k)},x_0,\ldots, x_0) + g(x_1,x_1^{(k)},x_1,\ldots, x_1) + g(x_0,x_1, x_2^{(k)}, \ldots, x_n^{(k)}) \\
& \quad \vdots \\
&\leq \sum\limits_{i=0}^{n} g(x_i^{(k)},x_i,\ldots, x_i) + g(x_0,x_1, \ldots, x_n)\\
&< \varepsilon  + g(x_0,x_1, \ldots, x_n).
\end{align*}
In a similar way, we have $g(x_0,x_1, \ldots, x_n) < \varepsilon  + g(x_0^{(k)},x_1^{(k)},\ldots, x_n^{(k)}).$\\
Therefore, $\big|g(x_0^{(k)},x_1^{(k)},\ldots, x_n^{(k)})-g(x_0,x_1, \ldots, x_n) \big| < \varepsilon$ as desired.

\end{proof}

\section{Fixed point theorems}
Fixed point theorems on a $G$-metric space have extensively been studied: see \cite{AKOH15} and references therein.
The interested reader can also refer to \cite{AN20131486, GABA17, GABA18, Khamsi2015}. In this section we generalize several fixed point theorems on the $g$-metric space under the topology established in Section 3 and Section 4.

The following result can be considered as a generalization of the Banach contractive mapping principle with respect to a $g$-metric space.
\begin{theorem}[Banach contractive mapping principle in a $g$-metric space] \label{T:Banach}
Let $(\Omega,g)$ be a complete $g$-metric space and let  $T:\Omega \longrightarrow \Omega$ be a mapping such that
there exists $\lambda \in [0, 1)$ satisfying
\begin{equation}\label{eq:contractive1}
g(T(x_0), T(x_1), \ldots, T(x_n)) \leq \lambda g(x_0, x_1, \ldots, x_n) \quad \text{for all } x_0, \ldots, x_n \in \Omega.
\end{equation}
Then $T$ has a unique fixed point in $\Omega$.
\end{theorem}

\begin{proof}
Let $y_0$ be an arbitrary point in $\Omega$.
Set $y_{k+1} = T(y_k)$ for all $k \in \N$.

{\bf{(Existence of a fixed point)}}
If $y_{m+1}=y_m$ for some $m\in \N$, then $y_m$ is a fixed point of $T$.
We assume that $y_{k+1} \neq y_k$ for all $ k \in \N$.
Then, by the condition \eqref{eq:contractive1} it follows that
\begin{equation}\label{eq:contractive-a}
g(y_{k+1},y_{k+2},y_{k+2},\ldots, y_{k+2}) \leq \lambda g(y_{k},y_{k+1},y_{k+1},\ldots, y_{k+1}) \quad \text{for all } k\in \N.
\end{equation}
So, by induction we have $g(y_{k},y_{k+1},y_{k+1},\ldots, y_{k+1}) \leq \lambda^k g(y_{0},y_{1},y_{1},\ldots, y_{1}) $, implying
$$ g(y_{k},y_{k+1},y_{k+1},\ldots, y_{k+1}) \longrightarrow 0 \quad \text{as } k\longrightarrow \infty.$$
Thus, $\{ y_k \}$ is a $g$-Cauchy sequence in $(\Omega,g)$ by Lemma \ref{lemma:g-cauchy}.
Since $(\Omega,g)$ is complete, there exists $y\in \Omega$ such that $\{y_k\} \overset{g}{\longrightarrow} y$.
It follows that
\begin{equation}\label{eq:contractive-b}
g(y_{k+1},T(y),T(y),\ldots,T(y))
\leq \lambda g(y_k, y, y, \ldots, y).
\end{equation}
As $k \longrightarrow \infty$, by Lemma \ref{lemma:g-continuous}
\begin{equation*}
g(y,T(y),T(y),\ldots,T(y))
\leq \lambda g(y, y, y, \ldots, y)=0.
\end{equation*}
Therefore, $T(y) = y$ by the positive definiteness for the $g$-metric.

{\bf{(Uniquness of a fixed point)}} Suppose that $y, \tilde{y}$ are distinct fixed points.
Then
\begin{align*}
g(\tilde{y},y,y,\ldots,y)
& = g(T(\tilde{y}),T(y),T(y),\ldots,T(y)) \\
& \leq \lambda g(\tilde{y},y,y,\ldots,y) \\
& <g(\tilde{y},y,y,\ldots,y),
\end{align*}
which is a contradiction. Thus, $y = \tilde{y}$.
\end{proof}

In fact, a weaker condition than the contractivity condition \eqref{eq:contractive1} can lead to the same conclusion as follows.
\begin{theorem}\label{T:Banach2}
Let $(\Omega,g)$ be a complete $g$-metric space and let  $T:\Omega \longrightarrow \Omega$ be a mapping such that
there exists $\lambda \in [0, 1)$ satisfying either
\begin{itemize}
\item[$(i)$] $g(T(x_0), x_1, \ldots, x_n) \leq \lambda g(x_0, x_1, \ldots, x_n)$ for all  $x_0, \ldots, x_n \in \Omega$
\item[or]
\item[ $(ii)$] $g(T(x), T(y), \ldots, T(y)) \leq \lambda g(x, y, \ldots, y)$  for all $x, y \in \Omega$.
\end{itemize}
Then $T$ has a unique fixed point in $\Omega$.
\end{theorem}
\begin{proof}
The proof of Theorem \ref{T:Banach2} is the same as the proof of Theorem  \ref{T:Banach}. Note that if the condition $(i)$ holds, then
it follows that
\begin{equation*}
g(T(x_0), T(x_1), \ldots, T(x_n)) \leq \lambda^{n+1} g(x_0, x_1, \ldots, x_n) \quad \text{for all } x_0, \ldots, x_n \in \Omega,
\end{equation*}
which is the contractive condition.
Also, it is noted that the condition $(ii)$ implies the inequalities \eqref{eq:contractive-a} and \eqref{eq:contractive-b}.
\end{proof}

We weaken the contractive conditions based on the notion of weak $\phi$-contractions.
The following are some families of control functions which are involved in establishing fixed point results. For more information about these families, see \cite{AKOH15}.
\begin{align*}
\mathcal{F}_{alt} & = \{ \phi:[0,\infty) \longrightarrow [0,\infty): \phi \text{ is continuous, non-decreasing, and } \phi^{-1}(\{0\})=\{0 \} \}, \\
\mathcal{F}'_{alt} & = \{ \phi:[0,\infty) \longrightarrow [0,\infty): \phi \text{ is lower semi-continuous and } \phi^{-1}(\{0\})=\{0 \} \}, \\
\mathcal{F}_{A} & = \{ \phi:[0,\infty) \longrightarrow [0,\infty): \phi \text{ is non-decreasing and } \phi^{-1}(\{0\})=\{0 \} \}.
\end{align*}

\begin{lemma}\cite[Lemma 2.3.3]{AKOH15}\label{lemma:FA}
If $\phi \in \mathcal{F}_A$ and $\{t_k\}_{k\in \N} \subseteq [0,\infty)$ is a  sequence such that $\phi(t_k) \longrightarrow 0$, then
$t_k \longrightarrow 0$.
\end{lemma}

\begin{lemma}\cite[Lemma 2.3.6]{AKOH15}\label{lemma:control3}
Let $\psi \in \mathcal{F}_{alt},~ \phi \in \mathcal{F}'_{alt}$ and let $\{t_k \}_{k\in \N} \subseteq [0,\infty)$ be a  sequence such that
\begin{equation}
\psi(t_{k+1}) \leq \psi(t_k) - \phi(t_k) \quad \text{for all } k \in \N.
\end{equation}
Then $t_k \longrightarrow 0$.
\end{lemma}

We show some fixed point results in $g$-metric spaces with weaker contractivity conditions involving the families of control functions.
\begin{theorem}
Let $(\Omega,g)$ be a complete $g$-metric space and let  $T:\Omega \longrightarrow \Omega$ be a self-mapping.
Assume that there exist two function $\psi \in \mathcal{F}_{alt}$ and $\phi \in  \mathcal{F}'_{alt}$ such that
\begin{equation} \label{eq:contractive-control1}
\psi(g(T(x), T(y), \ldots, T(y))) \leq \psi(g(x, y, \ldots, y)) - \phi(g(x, y, \ldots, y))
\end{equation}
for all $x,y \in \Omega$.
Then $T$ has a unique fixed point in $\Omega$. Furthermore, $T$ is a Picard operator.
\end{theorem}
\begin{proof}
Let $x_0$ be an arbitrary point in $\Omega$.
Set $x_{k+1} = T(x_k)$ for all $k \in \N$.

{\bf{(Existence of a fixed point)}}
If $x_{m+1}=x_m$ for some $m\in \N$, then $x_m$ is a fixed point of $T$.
We assume that $x_{k+1} \neq x_k$ for all $ k \in \N$.

Then, by the condition \eqref{eq:contractive-control1} it follows that for all $k\in \N$,
\begin{align*}
& \quad \psi(g(x_ {k+1}, x_ {k+2}, \ldots, x_ {k+2}))\\
&= \psi(g(T(x_k), T(x_{k+1}), \ldots, T(x_{k+1})) \\
& \leq \psi(g(T(x_k), T(x_{k+1}), \ldots, T(x_{k+1})) - \phi(g(T(x_k), T(x_{k+1}), \ldots, T(x_{k+1})).
\end{align*}
By Lemma \ref{lemma:control3}, it follows that
\begin{equation}
\lim_{k\rightarrow \infty} g(x_ {k+1}, x_ {k+2}, \ldots, x_ {k+2}) = 0.
\end{equation}
By Lemma \ref{lemma:g-cauchy}, $\{x_k\}_{k\in \N}$ is a $g$-Cauchy sequence.
Since $(\Omega,g)$ is complete, there exists $x \in \Omega$ such that $\{x_k\} \overset{g}{\longrightarrow} x$.
It follows that
\begin{align*}
\psi(g(x_{k+1},T(x),T(x),\ldots,T(x)))
&= \psi(g(T(x_k), T(x), T(x), \ldots, T(x))) \\
&\leq \psi(g(x_k, x, \ldots, x)) - \phi(g(x_k, x, \ldots, x))\\
& \leq \psi(g(x_k, x, \ldots, x)).
\end{align*}
Taking the limit as $k \longrightarrow \infty$ on the both sides, by the continuity of $g$ and $\psi$, we have
\begin{align*}
\lim_{k \rightarrow \infty} \psi(g(x_{k+1},T(x),T(x),\ldots,T(x)))
& \leq \lim_{k \rightarrow \infty} \psi(g(x_k, x, \ldots, x)) \\
& = \psi(g(x, x, \ldots, x))=0.
\end{align*}
By Lemma \ref{lemma:FA} and the continuity of $g$, we have $g(x,T(x),T(x),\ldots,T(x))=0$.
Therefore, $T(x) = x$ by the condition $(g1)$.

{\bf{(Uniquness of a fixed point)}} Suppose that $x, \tilde{x}$ are distinct fixed points.
Then $g(\tilde{x}, x, \ldots, x)>0$. Since $\phi \in \mathcal{F}'_{alt}$, it holds that $\phi (g(\tilde{x}, x, \ldots, x))>0$.
By the contractivity condition, we have that
\begin{align*}
\psi(g(\tilde{x},x,x,\ldots,x))
& = \psi(g(T(\tilde{x}),T(x),T(x),\ldots,T(x))) \\
& \leq \psi(g(\tilde{x}, x, \ldots, x)) - \phi(g(\tilde{x}, x, \ldots, x))\\
& < \psi(g(\tilde{x},x,x,\ldots,x)),
\end{align*}
which is a contradiction. Thus, $x = \tilde{x}$.
\end{proof}

\begin{definition}
Let $(\Omega, g)$ be a $g$-metric space. A mapping $T: \Omega \to \Omega$ is said to be \emph{weak $g$-contractive} if
\begin{equation*}
g(T(x_0), T(x_1), \ldots, T(x_n)) < g(x_0, x_1, \ldots, x_n)
\end{equation*}
for which any two of $x_0, \ldots, x_n \in \Omega$ are distinct.
\end{definition}

\begin{proposition}\label{prop:continuity}
Let $(\Omega, g)$ be a $g$-metric space. Suppose that $T: \Omega \longrightarrow \Omega$ is a weak $g$-contractive function. Then the function $f: \Omega \longrightarrow \mathbb{R}_{+}$ given by $f(x) = g(x, T(x), \ldots, T(x))$ is continuous.
\end{proposition}

\begin{proof}
Let $x \in \Omega.$ We need to show that for any $\varepsilon > 0$ there exists $\delta > 0$ such that
$\big| f(B_{g}(x, \delta)) - g(x, T(x), \ldots, T(x)) \big| < \varepsilon$. We let $\displaystyle{\delta = \frac{\varepsilon}{n+1}}$. For $y \in B_{g}(x, \delta),$ we first assume that $g(y, T(y), \ldots, T(y)) \leq g(x, T(x), \ldots, T(x)).$ Then
\begin{flalign*}
&|g(x, T(x), \ldots, T(x))-g(y, T(y), \ldots, T(y))|\\
&= g(x, T(x), \ldots, T(x))- g(y, T(y), \ldots, T(y))\\
& \leq g(x, y, \ldots, y) + g(y, T(x), \ldots, T(x))- g(y, T(y), \ldots, T(y)) \quad \text{(by Theorem \ref{thm:g-metric:basic-properties} (2))}\\
& \leq g(x, y, \ldots, y) + g(T(y), T(x), \ldots, T(x)) \quad \text{(by Theorem \ref{thm:g-metric:basic-properties} (2))}\\
& \leq g(x, y, \ldots, y) + g(y, x, \ldots, x) \quad \text{(by the weak $g$-contractivity of $T$)}\\
& \leq g(x, y, \ldots, y) + n g(x, y, \ldots, y) \quad \text{(by Theorem \ref{thm:g-metric:basic-properties} (3))}\\
&< (1+n)\delta = \varepsilon
\end{flalign*}
In a similar way, it can be proved that $|g(x, T(x), \ldots, T(x))-g(y, T(y), \ldots, T(y))| < \varepsilon$ holds when $g(y, T(y), \ldots, T(y)) \geq g(x, T(x), \ldots, T(x)).$ Hence, $f$ is continuous.
\end{proof}

\begin{theorem}
Let $T$ be a weak $g$-contractive mapping on a $g$-compact $g$-metric space $(\Omega, g)$. Then $T$ has a unique fixed point.
\end{theorem}
\begin{proof}
By Proposition \ref{prop:continuity} the function $f: \Omega \ni x \mapsto g(x, T(x), \dots, T(x)) \in \mathbb{R}_{+}$ is continuous. Since $\Omega$ is $g$-compact, the continuous function $f$ attains its minimum at some $\bar{x} \in \Omega$. If $\bar{x} \neq T(\bar{x})$, then
\begin{displaymath}
\begin{split}
g(\bar{x}, T(\bar{x}), \ldots, T(\bar{x})) &= \min_{x \in \Omega} g(x, T(x), \dots, T(x)) \\
& \leq g(T(\bar{x}), T(T(\bar{x})), \dots, T(T(\bar{x}))) \\
& < g(\bar{x}, T(\bar{x}), \ldots, T(\bar{x})),
\end{split}
\end{displaymath}
which is a contradiction. So $\bar{x}$ is a fixed point of $T$. The uniqueness argument follows exactly same as in the proof of Theorem \ref{T:Banach}.
\end{proof}

We next generalize \'Ciri\'c fixed point theorem \cite{Cir74} in a $g$-metric space.\\
Let $(\Omega,g)$ be a $g$-metric space and $T: \Omega \longrightarrow \Omega$ a map.
For each $x \in \Omega$, we denote
$$O(x, N)  = \{x,~T(x),~ T^{2}(x), \ldots ,~T^{N}(x)\} \text{~and~} O(x, \infty)  = \{x,~T(x),~T^{2}(x),\ldots\} ,$$
where $T^{k+1}=T \circ T^k$ for all $k\in \N$ and $T^{0}$ is the identity mapping on $\Omega$.

\begin{definition}
\begin{itemize}
\item[(1)] A $g$-metric space $\Omega$ is said to be \emph{$T$-orbitally $g$-complete} if every $g$-Cauchy sequence contained in $O(x, \infty)$ for some $x \in \Omega$ is $g$-convergent in $\Omega.$
\item[(2)] A mapping $T: \Omega \longrightarrow \Omega$ is called a \emph{$g$-quasi-contraction} if there exists $\lambda \in [0, 1)$ such that for all $x_0,\ldots,x_{n} \in \Omega.$
\begin{align*}
g(T(x_{0}),\ldots,T(x_{n}))
 \leq &~ \frac{\lambda}{n} \max \Big[\{g(x_{0},\ldots,x_{n})\} \\
& \cup \{g(x_{i},T(x_{j}),\ldots,T(x_{j})) : i,j = 0,\ldots,n\} \Big].
\end{align*}
\end{itemize}
\end{definition}
For $A \subseteq \Omega,$ we denote $\hbox{sup}\{g(a_{0}, \ldots , a_{n}):a_{0},\ldots,a_{n} \in A\}$ by $s(A).$
\begin{lemma}\label{lemma:g-quasi-contraction}
Suppose that $T: \Omega \longrightarrow \Omega$ is a $g$-quasi-contraction on a $g$-metric space $(\Omega,g).$ Then for each $x \in \Omega$ the following inequalities hold.
\begin{itemize}
  \item[(1)] $\displaystyle{ g(T^{k_{0}}(x),\ldots,T^{k_{n}}(x)) \leq \frac{\lambda}{n} s(O(x, N)) }$ for all $k_{0},\ldots,k_{n} \in \{1,\ldots,N\}$.
  \item[(2)] $\displaystyle{ s(O(x, \infty)) \leq \frac{n}{1-\lambda}g(x,T(x),\ldots,T(x)) }$.
\end{itemize}
\end{lemma}

\begin{proof}
(1) Let $x \in \Omega.$ Since $\{T^{k_{0}}(x),T^{k_{0}-1}(x),\ldots,T^{k_{n}}(x),T^{k_{n}-1}(x)\}$ is a subset of $O(x, N)$ and the mapping $T$ is a $g$-quasi-contraction, there exists $\lambda \in [0, 1)$ such that
\begin{align*}
g(T^{k_{0}}(x),\ldots,T^{k_{n}}(x))
&= g(TT^{k_{0}-1}(x),\ldots,TT^{k_{n}-1}(x))\\
& \leq  \frac{\lambda}{n} \hbox{max}\Big[\{g(T^{k_{0}-1}(x),\ldots,T^{k_{n}-1}(x))\}\\
&\quad \cup \{g(T^{k_{i}-1}(x),T^{k_{j}}(x),\ldots,T^{k_{j}}(x)) : i,j = 0,\ldots,n\}\Big]\\
&\leq \frac{\lambda}{n} s(O(x, N)).
\end{align*}

(2) Let $x \in \Omega.$
Since the sequence $\{s(O(x, N))\}_{N\in \N}$ is monotonically increasing, $s(O(x, \infty))=\hbox{sup}\{s(O(x, N)):N \in \mathbb{N}\}.$ For a fixed positive integer $N_{0},$ the statement (1) implies that there exist $k_{1},k_{2},\ldots,k_{n} \in \{0,1,\ldots,N_{0}\}$ such that \\
$g(x,T^{k_{1}}(x), \ldots , T^{k_{n}}(x))=s(O(x, N_{0})).$ Without loss of generality, we can assume that $k_{1} \leq k_{2} \leq \cdots \leq k_{n}.$ If $k_{n}=0$ (i.e. $k_{i}=0$ for all $i$), then $s(O(x, N_{0}))=g(x,x, \ldots , x)=0.$ Suppose that there exists $1 \leq j  \leq n $ such that $k_{j} \neq 0$ and $k_{j-1} = 0.$ Then by Theorem \ref{thm:g-metric:basic-properties} (4) and the statement (1) we have
\begin{align*}
g(x,T^{k_{1}}(x), \ldots , & T^{k_{n}}(x))
\leq g(x,T(x), \ldots, T(x)) + \sum\limits_{i=1}^{n}g(T^{k_{i}}(x),T(x), \ldots, T(x))\\
& = jg(x,T(x), \ldots, T(x)) + \sum\limits_{i=j}^{n}g(T^{k_{i}}(x),T(x), \ldots, T(x))\\
& \leq jg(x,T(x), \ldots, T(x)) + (n-j+1) \frac{\lambda}{n} s(O(x, N_{0}))\\
& = jg(x,T(x), \ldots, T(x)) + (n-j+1) \frac{\lambda}{n} g(x,T^{k_{1}}(x), \ldots , T^{k_{n}}(x))\\
\end{align*}
Thus, it follows that
\begin{align*}
s(O(x, N_{0}))
&=g(x,T^{k_{1}}(x), \ldots , T^{k_{n}}(x)) \\
&\leq \dfrac{j}{1-\frac{n-j+1}{n}\lambda}g(x,T(x), \ldots, T(x)) \\
& \leq \dfrac{n}{1-\lambda}g(x,T(x), \ldots, T(x)).
\end{align*}
Since $N_{0}$ is arbitrary, $s(O(x, \infty)) \leq \dfrac{n}{1-\lambda}g(x,T(x),\ldots,T(x)).$
\end{proof}

\begin{theorem}[\'Ciri\'c fixed point theorem in a $g$-metric space]
Let $\Omega$ be a $g$-metric space. Suppose that $\Omega$ is $T$-orbitally $g$-complete and $T: \Omega \longrightarrow \Omega$ is a $g$-quasi-contraction. Then the following are true.
\begin{itemize}
\item[(1)] $T$ has a unique fixed point $y$ in $\Omega.$
\item[(2)] $\{T^{N}(x)\} \overset{g}{\longrightarrow}  y$ as $N \longrightarrow \infty$.
\item[(3)] $g(T^{N}(x), y, \ldots, y) \leq \dfrac{\lambda^N}{n^{N-1}(1-\lambda)}g(x,T(x),\ldots,T(x)).$
\end{itemize}
\end{theorem}
\begin{proof}
\begin{itemize}
\item[(2)]
Let $x \in \Omega.$ Since $T$ is a $g$-quasi-contraction, by Lemma \ref{lemma:g-quasi-contraction} (1) it follows that
\begin{align*}
g(T^{k_{0}}(x), \ldots, T^{k_{n}}(x))
& = g(TT^{k_{0}-1}(x), T^{k_{1}-k_{0}+1}T^{k_{0}-1}(x), \ldots, T^{k_{n}-k_{0}+1}T^{k_{0}-1}(x))\\
& \leq \frac{\lambda}{n} s(O(T^{k_{0}-1}(x), k_{n}-k_{0}+1))
\end{align*}
for positive integers $k_{0}, k_{1}, \ldots, k_{n}$ with $k_{0} < k_{1} < \cdots < k_{n}$.
By Lemma \ref{lemma:g-quasi-contraction} (1), there exist $\ell_{1}, \ldots, \ell_{n} \in \{0, \ldots, k_{n}-k_{0}+1\}$ (without loss of generality, we assume that $\ell_{1} \leq \cdots \leq \ell_{n}$) such that
\begin{equation*}
s(O(T^{k_{0}-1}(x), k_{n}-k_{0}+1)) = g(T^{k_{0}-1}(x), T^{l_{1}}T^{k_{0}-1}(x), \ldots, T^{l_{n}}T^{k_{0}-1}(x)).
\end{equation*}
Then by Lemma \ref{lemma:g-quasi-contraction} (1), we have
\begin{align*}
 g(T^{k_{0}-1}(x), &T^{\ell_{1}}T^{k_{0}-1}(x), \ldots, T^{\ell_{n}}T^{k_{0}-1}(x))\\
& =g(TT^{k_{0}-2}(x), T^{\ell_{1}+1}T^{k_{0}-2}(x), \ldots, T^{\ell_{n}+1}T^{k_{0}-2}(x)) \\
& \leq \frac{\lambda}{n} s(O(T^{k_{0}-2}(x), \ell_{n}+1)) \leq \frac{\lambda}{n} s(O(T^{k_{0}-2}(x), k_{n}-k_{0}+2)).
\end{align*}
By repeating process, we eventually obtain the following inequalities:
\begin{align*}
g(T^{k_{0}}(x), \ldots, T^{k_{n}}(x))
&\leq \frac{\lambda}{n} s(O(T^{k_{0}-1}(x), k_{n}-k_{0}+1))\\
& \leq \bigg(\frac{\lambda}{n} \bigg)^{2} s(O(T^{k_{0}-2}(x), k_{n}-k_{0}+2))\\
& \quad \vdots\\
& \leq \bigg(\frac{\lambda}{n} \bigg)^{k_{0}} s(O(x, k_{n})).
\end{align*}
Then it follows from Lemma \ref{lemma:g-quasi-contraction} (2) that
\begin{equation} \label{upperbound}
g(T^{k_{0}}(x), \ldots, T^{k_{n}}(x)) \leq \bigg(\frac{\lambda}{n} \bigg)^{k_{0}}\frac{n}{1-\lambda}g(x,T(x),\ldots,T(x)).
\end{equation}
The sequence of iterates $\{T^{N}(x)\}$ is $g$-Cauchy because $\displaystyle{\Big(\frac{\lambda}{n} \Big)^{k_{0}}}$ tends to $0$ as $k_{0} \longrightarrow \infty.$ Therefore, since $\Omega$ is $T$-orbitally $g$-complete, $\{T^{N}(x)\}$ has the $g$-limit $y$ in $\Omega.$

\item[(1)] (\textbf{Existence of a fixed point})
We shall show that the $g$-limit $y$ is a fixed point under $T.$ Let us consider the following inequalities:
 \begin{align*}
g(y, T(y), \ldots, T(y))
& \leq g(y, T^{N+1}(y), \ldots, T^{N+1}(y)) + g(TT^{N}(y), T(y), \ldots, T(y)) \\
& \leq g(y, T^{N+1}(y), \ldots, T^{N+1}(y)) + \frac{\lambda}{n} \hbox{max}\Big\{g(T^{N}(y), y, \ldots, y),\\
& \qquad g(T^{N}(y), T^{N+1}(y), \ldots, T^{N+1}(y)),~ g(y, T(y), \ldots, T(y)),\\
& \qquad g(T^{N}(y), T(y), \ldots, T(y)),~ g(y, T^{N+1}(y), \ldots, T^{N+1}(y)) \Big\} \\
& \leq g(y, T^{N+1}(y), \ldots, T^{N+1}(y)) + \frac{\lambda}{n} \Big(g(T^{N}(y), y, \ldots, y)\\
& \quad +g(T^{N}(y), T^{N+1}(y), \ldots, T^{N+1}(y))~ + g(y, T(y), \ldots, T(y)) \\
& \quad +g(y, T^{N+1}(y), \ldots, T^{N+1}(y)) \Big) \quad \text{(by Theorem \ref{thm:g-metric:basic-properties} (2))}.
\end{align*}
Then for every positive integer $N,$ we have
\begin{align*}
g(y, T(y), \ldots, T(y))
& \leq \frac{\lambda}{n-\lambda} \Big[ g(T^{N}(y), y, \ldots, y) + g(T^{N}(y), T^{N+1}(y), \ldots, T^{N+1}(y))\\
& \qquad + \Big(\frac{n}{\lambda}+1 \Big)g(y, T^{N+1}(y), \ldots, T^{N+1}(y)) \Big].
\end{align*}
Note that for any $x \in \Omega,$ $\{T^{N}(x)\} \overset{g}{\longrightarrow} y.$ Thus $g(y, T(y), \ldots, T(y))=0,$ i.e. $T(y)=y.$ Therefore, $y$ is a fixed point of $T.$

(\textbf{Uniqueness of a fixed point})
Suppose that $y$ and $\widetilde{y}$ are fixed points under $T,$ i.e. $T(y)=y$ and $T(\widetilde{y})=\widetilde{y}.$ The $g$-quasi-contractivity of $T$ gives rise to the following inequality:
\begin{align*}
g(\widetilde{y}, y, \ldots, y)
&= g(T(\widetilde{y}), T(y), \ldots, T(y))\\
& \leq \frac{\lambda}{n} \max\Big\{g(\widetilde{y}, y, \ldots, y), g(\widetilde{y}, T(\widetilde{y}), \ldots, T(\widetilde{y})), g(y, T(y), \ldots, T(y)),\\
& \qquad \qquad \quad  g(\widetilde{y}, T(y), \ldots, T(y)), g(y, T(\widetilde{y}), \ldots, T(\widetilde{y})) \Big\} \\
& \leq \frac{\lambda}{n} \max \Big\{g(\widetilde{y}, y, \ldots, y), g(y, \widetilde{y}, \ldots, \widetilde{y}) \Big\} \\
& \leq \frac{\lambda}{n} \hbox{max}\Big\{g(\widetilde{y}, y, \ldots, y), n g(\widetilde{y}, y, \ldots, y) \Big\} \quad \text{(by Theorem \ref{thm:g-metric:basic-properties} (3))} \\
& = \lambda g(\widetilde{y}, y, \ldots, y).
\end{align*}
Since $0 \leq \lambda <1$, it holds that $g(\widetilde{y}, y, \ldots, y)=0.$ Therefore, $y = \widetilde{y}$ as desired.

\item[(3)] Taking the limit as $k_{1} \longrightarrow \infty$ on the both side of \eqref{upperbound}, we can obtain the inequality
\begin{equation*}
g(T^{k_0}(x), y, \ldots, y) \leq \bigg(\frac{\lambda}{n}\bigg)^{k_0} \bigg( \frac{n}{1-\lambda} \bigg) g(x,T(x),\ldots,T(x)).
\end{equation*}

\end{itemize}

\end{proof}


\section*{Acknowledgments}
The work of S. Kim was supported by the National Research Foundation of Korea (NRF) grant funded by the Korea government (MIST) (No. NRF-2018R1C1B6001394).
This work of H. Choi was partially supported  by Shanghai Sailing Program under Grant 16YF1407700, and National Nature Science Foundation of China (NSFC) under Grant No. 61601290.

\section*{References}
\bibliographystyle{plain}

\begin{thebibliography}{99}

\bibitem{AKOH15} R. P. Agarwal, E. Karapinar, D. O'Regan, and A. F. Roldan-Lopez-de-Hierro. \emph{Fixed Point Theory in Metric Type Spaces}. Springer International Publishing, 2015.

\bibitem{AN20131486} T. V. An, N. V. Dung, and V. T. L. Hang. A new approach to fixed point theorems on {G}-metric spaces. \emph{Topol. Appl.}, 160(12):1486--1493, 2013.

\bibitem{Cir74} Lj. B. \'{C}iri\'{c}. A Generalization of {B}anach's Contraction Principle. \emph{Proc. Amer. Math. Soc.}, 45(2):267--273, 1974.

\bibitem{dhage92} B. C. Dhage. Generalized metric spaces and mapping with fixed points. \emph{Bull. Calcutta Math. Soc.}, 84:329--336, 1992.

\bibitem{DRAGER2007929} L. D. Drager, J. M. Lee, and C. F. Martin. On the geometry of the smallest circle enclosing a finite set of points. \emph{J. Franklin Inst.}, 344(7):929--940, 2007.

\bibitem{Frechet06} M. Fr\'{e}chet. Sur quelques points du calcul fonctionnel. \emph{Rendiconti del Circolo Matematico di Palermo}, 22(1):1--74, 1906.

\bibitem{GABA17} Y. U. Gaba. Fixed point theorems in {G}-metric spaces. \emph{J. Math. Anal. Appl.}, 455(1):528--537, 2017.

\bibitem{GABA18} Y. U. Gaba. Fixed points of rational type contractions in {G}-metric spaces. \emph{Cogent Mathematics $\&$ Statistics}, 5(1):1444904, 2018.

\bibitem{Gahler63} S. Gahler. 2-metrische {R}aume und ihre topologische {S}trukture. \emph{Math. Nachr}, 26:115--148, 1963.

\bibitem{Gahler66} S. Gahler. Zur geometric 2-metrische {R}aume. \emph{Reevue Roumaine de Math. Pures et Appl.}, XI:"664--669, 1966.

\bibitem{HCW88} K. S. Ha, Y. J. Cho, and A. White. Strictly convex and strictly 2-convex 2-normed spaces. \emph{Math. Japonica}, 33(3):375--384, 1988.

\bibitem{Khamsi2015} M. A. Khamsi, Generalized metric spaces: A survey. \emph{Journal of Fixed Point Theory and Applications}, 17(3):455--475, 2015.

\bibitem{Mordukhovich2013} B. Mordukhovich,  N. M. Nam, and C. Villalobos. The smallest enclosing ball problem and the smallest intersecting ball problem: existence and uniqueness of solutions. \emph{Optim. Lett.}, 7(5):839--853, 2013.

\bibitem{MS03} Z. Mustafa and B. Sims. Some remarks concerning {D}-Metric Spaces. In \emph{Proceedings of the International Conference on Fixed Point Theorey and Applications, Valencia (Spain)}, pages 189--198, 2003.

\bibitem{MS06} Z. Mustafa and B. Sims, A new approach to generalized metric spaces. \emph{J. Nonlinear Convex Anal.}, 7(2):289--297, 2006.

\bibitem{NRS05} S. V. R. Naidu, K. P. R. Rao, and N. Srinivasa Rao. On the concepts of balls in a {D}-metric space. \emph{Int. J. Math. Math. Sci.}, 2005(1):133--141, 2005.

\bibitem{PR91} M. Padberg and G. Rinaldi. A Branch-and-Cut Algorithm for the Resolution of Large-Scale Symmetric Traveling Salesman Problems. \emph{SIAM Review}, 33(1):60--100, 1991.

\bibitem{sylvester1857} J. J. Sylvester. A question in the geometry of situation. \emph{Q. J. Pure Appl. Math.}, 79(1), 1857.

\bibitem{Verblunsky1951} S. Verblunsky. On the Shortest Path Through a Number of Points. \emph{Proc. Am. Math. Soc.}, 2(6):904--913, 1951.

\bibitem{Welzl91} E. Welzl. Smallest enclosing disks (balls and ellipsoids). In \emph{New Results and New Trends in Computer Science}, pages 359--370, Berlin, Heidelberg, 1991. Springer Berlin Heidelberg.

\end{thebibliography}

\end{document}